\documentclass[12pt]{amsart}

\usepackage{longtable}
\usepackage{array}
\usepackage{afterpage}
\usepackage{float}
\usepackage{longtable}
\usepackage{pdflscape}
\usepackage[numbers,sort&compress]{natbib}
\usepackage{psfrag}
\usepackage[unicode=true,
  linktocpage,
  linkbordercolor={0.5 0.5 1},
  citebordercolor={0.5 1 0.5},
  linkcolor=blue]{hyperref}
\usepackage{setspace} 
\usepackage{palatino}
\usepackage[mathcal]{euler}
\usepackage{amssymb,amsmath,amsthm,enumerate,slashed}

\usepackage[shortlabels]{enumitem} 

\usepackage[latin1]{inputenc}
\usepackage[T1]{fontenc}
\usepackage{xcolor}
\usepackage{mathdots} 
\usepackage{ytableau,tikz,varwidth}
\usetikzlibrary{calc}
\usetikzlibrary{shapes.geometric}
\usepackage{mathtools} 
\usepackage{tikz-cd}
\usepackage[left=1.25in, bottom=1in, right=1in, top=1in, paperwidth=8.5in, paperheight=11in]{geometry}
\usepackage{graphicx}
\usepackage{epstopdf,epsfig}
\usepackage{xy}	
\xyoption{all}
\usepackage{marginnote}
\usepackage[ddmmyyyy]{datetime}
\usepackage[normalem]{ulem}
\allowdisplaybreaks 

\makeatletter
\newcommand*\bigcdot{\mathpalette\bigcdot@{.5}}
\newcommand*\bigcdot@[2]{\mathbin{\vcenter{\hbox{\scalebox{#2}{$\m@th#1\bullet$}}}}}
\makeatother

\newcommand{\Q}{\mathbb{Q}}

\newcommand{\so}{\qquad\text{ so }\qquad}
\newcommand{\andd}{\qquad\text{ and }\qquad}

\newcommand{\orr}{\qquad\text{ or }\qquad}

\newcommand{\svw}[1]{\textcolor{black}{#1}}
\newcommand{\svwj}[1]{\textcolor{black}{#1}}

\swapnumbers
\numberwithin{equation}{section}
\theoremstyle{plain}

\newtheorem*{theorem*}{Theorem}
\newtheorem*{lemma*}{Lemma}
\newtheorem{theorem}[equation]{Theorem}
\newtheorem{lemma}[equation]{Lemma}
\newtheorem{proposition}[equation]{Proposition}
\newtheorem{corollary}[equation]{Corollary}
\newtheorem{conjecture}[equation]{Conjecture}
\theoremstyle{definition}
\newtheorem{definition}[equation]{Definition}
\newtheorem*{definition*}{Definition}
\newtheorem{example}[equation]{Example}
\newtheorem{notation}[equation]{Notation}
\newtheorem*{notation*}{Notation}
\newtheorem{remark}[equation]{Remark}
\newtheorem*{remark*}{Remark}

\theoremstyle{remark}

\numberwithin{equation}{section}

\begin{document}

\title[Schur-positivity for
generalized nets]{Schur-positivity for
generalized nets}

\author{Ethan Shelburne}
\address{\svw{
 Department of Mathematics,
 University of British Columbia,
 Vancouver BC V6T 1Z2, Canada}}
\email{\svw{emshelburne@math.ubc.ca}}

\author{{Stephanie van Willigenburg}}
\address{
 \svw{Department of Mathematics,
 University of British Columbia,
 Vancouver BC V6T 1Z2, Canada}}
\email{\svw{steph@math.ubc.ca}}

\thanks{\svw{{Both} authors were supported in part by the \svw{Natural} Sciences and Engineering Research Council of Canada.}}





\begin{abstract}
A {graph} is Schur-positive if its chromatic symmetric function expands nonnegatively
in the Schur basis. All claw-free graphs are conjectured to be Schur-positive. We
introduce a combinatorial object corresponding to a graph G, called a special rim hook
G-tabloid, which is a variation on the special rim hook tabloid. These objects can be
employed to compute any Schur coefficient of the chromatic symmetric function of a
graph.  We construct sign-reversing maps on these special rim
hook G-tabloids to obtain a recurrence relation for the Schur coefficients of a family of claw-free graphs called generalized nets, then we prove the entire family is Schur-positive. We subsequently determine an analogous recurrence relation for another, similar family of claw-free graphs. Thus, we demonstrate a new method for proving Schur-positivity of chromatic symmetric functions, which has the potential to be applied to make further progress toward the
aforementioned conjecture.

\end{abstract}

\maketitle



\section{Introduction}
\label{sect:intro}

In 1995, Richard Stanley introduced the chromatic symmetric function in \cite{stan_95}, extending the idea of counting proper colorings of graphs to the algebra of symmetric functions. 
Symmetric functions which expand nonnegatively in the $e$-basis or Schur basis often exhibit special representation theoretical or algebraic properties, such as intricate combinatorial interpretations of their coefficients. For example, all symmetric functions which expand nonnegatively in the Schur basis arise as the Frobenius image of some representation of the symmetric group, as discussed in \cite{sym_group}. A natural goal stated in \svwj{\cite[p. 186]{stan_95}} is to classify exactly which graphs \svw{have a chromatic symmetric function that is} Schur-positive or $e$-positive. In particular, the following three major open conjectures emerged from Stanley's studies of chromatic symmetric functions.

\begin{enumerate}

\item (The Stanley-Stembridge Conjecture, \cite{stan_93}): All claw-free incomparability graphs are $e$-positive.

\item (The Nonisomorphic Tree Conjecture, \cite{stan_95}): No two nonisomorphic trees have the same chromatic symmetric function.

\item (The Claw-free Conjecture, \cite{stan_98}): All claw-free graphs are Schur-positive.

\end{enumerate}

Subfamilies of graphs have been shown to satisfy the Stanley-Stembridge Conjecture in \svw{\cite{centered_at_vertex, Dahl18, GebSag01},} among others, \svwj{and a proof has recently been announced \cite{hikita}.} The Nonisomorphic Tree Conjecture has been computationally confirmed on trees of up to 29 vertices in \cite{trees29}. Other results toward this conjecture can be found in \svw{\cite{foley22, mar_mor, OrSc14},} for example. Vesselin Gasharov made significant headway toward the Claw-free Conjecture by proving that all claw-free incomparability graphs are Schur-positive in \cite{gash_96}. Moreover, it was shown in \cite{Kalis14} that all coefficients of chromatic symmetric functions corresponding to partitions of ``hook'' shapes are nonnegative. Most recently, David Wang and Monica Wang introduced a combinatorial formula which yields the Schur coefficients of chromatic symmetric functions, and used it to prove that certain graphs are not Schur-positive in \cite{WW2020}. 
 

Due to Gasharov's result, the task of proving the Claw-free Conjecture has been reduced to showing the Schur-positivity of claw-free graphs which are not incomparability graphs. We focus on a particular family of graphs satisfying these properties known as generalized nets, which consist of complete graphs with degree one vertices appended. We prove that all such graphs are Schur-positive, thus making progress toward the Claw-free Conjecture. We then extend our methods to prove a related result for another family of claw-free, non-incomparability graphs, which are generalized nets with one vertex added.



In order to achieve these results, we introduce a new version of the combinatorial objects known as special rim hook tabloids. We then reinterpret the formula from \cite{WW2020} in terms of these objects, and construct sign-reversing maps on them to obtain recurrence relations for the Schur coefficients of the chromatic symmetric functions of the \svw{two} families of graphs we consider. Likewise, we use similar methods to find an explicit formula for certain Schur coefficients of generalized nets. Accordingly, we develop new methods that can potentially be applied to other families of graphs in order to advance the pursuit of a proof of the Claw-free Conjecture.

Section \ref{sect:prelims} covers the necessary background.
Section \ref{sect:SRH_G_TABS} introduces special rim hook G-tabloids and elucidates their role in computing Schur coefficients for chromatic symmetric functions. Section \ref{sect:gen_nets} presents a proof of the Schur-positivity of generalized nets. Finally, Section \ref{sect:gen_spiders} establishes a recurrence relation on the Schur coefficients corresponding to the modified version of generalized nets on which we focus.


\section*{Acknowledgments} \svwj{The authors would like to thank the referee for their careful reading and thoughtful comments.}

\section{Preliminaries}
\label{sect:prelims}
 A \textit{partition} of $n\ge 0$ is a sequence of weakly decreasing positive integers $\lambda=(\lambda_1,\dots,\lambda_k)$ whose sum is $n$.
 The \textit{length} of $\lambda$ is given by $\ell(\lambda)=k$ and the \textit{size} of $\lambda$ is given by $|\lambda|=n$. In the case where $n=0$, we say $\lambda$ is the \textit{empty partition}. We use exponents to denote repeated integers in a partition. For example, $\lambda=(3,3,2,2,2,1)=(3^2,2^3,1)$. The \textit{diagram} of a partition $\lambda$ of $n$ is an array of $n$ boxes (called \textit{cells}) in left-justified rows such that row $i$ contains $\lambda_i$ boxes, where the rows are indexed from top to bottom and the columns are indexed from left to right. Below, we depict the diagram of the partition $(5,4,3,3,2)$.

 \begin{figure}[H]
 \begin{tikzpicture}
 \node[scale=.6] at (0,0) (a) {\ydiagram{5,4,3,3,2}};
  \end{tikzpicture}
\label{fig:young_diagram_example}
 \end{figure}

 A \textit{composition} of $n\ge 0$ is a sequence of positive integers $\kappa=[\kappa_1,\dots,\kappa_j]$ whose sum is $n$. The \textit{length} of $\kappa$ is given by $\ell(\kappa)=j$ and the \textit{size} of $\kappa$ is given by $|\kappa|=n$. In the case where $n=0$, we say $\kappa$ is the \textit{empty composition}.
 Given a composition $\kappa$, we use $\Lambda(\kappa)$ to denote the partition obtained by arranging the integers in $\kappa$ in weakly decreasing order. Again, we sometimes use exponents to denote repeated integers in a composition.

We now let $\mathbf{x}=\{x_1,x_2,x_3,\dots\}$ be a countably infinite set of commuting variables and consider the \textit{algebra of formal power series} in variables $\mathbf{x}$ over the rational numbers, which we denote by $\Q[[\mathbf{x}]]$. We say $f(\mathbf{x})$ is symmetric if it is invariant under any permutation of the variables $\mathbf{x}$. The subspace
\[
\text{Sym}(\mathbf{x})=\{f\in \Q[[\mathbf{x}]]\,\mid\,f\text{ is symmetric}\}
\]
has the structure of an algebra and is called the\textit{ algebra of symmetric functions}.


We focus on the classical Schur basis for Sym$(\mathbf{x})$, which we introduce using semistandard Young tableau.     A \textit{semistandard Young tableau} (SSYT) of \textit{shape} $\lambda$ is a filling $Q$ of the cells of the diagram $\lambda$ with positive integers such that rows weakly increase from left to right and columns strictly increase from top to bottom.  Given a semistandard Young tableau $Q$, we define the \textit{weight} of $Q$ to be
    \[
    \text{wt}(Q)=x_1^{\#1s}x_2^{\#2s}x_3^{\#3s}\cdots.
    \]

\begin{example}
 Below we portray several examples of SSYTs of shape (4,2,1). 
 
 \begin{figure}[H]
\begin{align*}
\begin{ytableau}
       1 &  2 & 3  & 4 \\
       5  & 6  \\
        7
\end{ytableau}
\qquad
\begin{ytableau}
       1 &  1 &  2 & 4 \\
       2  & 3  \\
        4
\end{ytableau}
\qquad
\begin{ytableau}
       2 &  2 & 2  & 7 \\
       3  & 3  \\
        5
\end{ytableau}
\qquad
\begin{ytableau}
       1 &  2 & 5  & 5 \\
       5  & 5  \\
        6
\end{ytableau}
\end{align*}

\label{fig:SSYTs}
\end{figure}
 
 From left to right, these SSYTs have weights
 \begin{align*}
 \label{eq:SSYTweights}
 x_1x_2x_3x_4x_5x_6x_7,\qquad x_1^2x_2^2x_3x_4^2,\qquad x_2^3x_3^2x_5x_7,\andd x_1x_2x_5^4x_6.
 \end{align*}

\end{example}

Given some partition $\lambda$, the \textit{Schur function} associated to $\lambda$ is
\[
s_{\lambda}=\sum_{Q}\text{wt}(Q)
\]
where the sum spans over all semistandard Young tableaux $Q$ of shape $\lambda$. All Schur functions are symmetric and
\[
\mathbf{s}=\{s_{\lambda}\,\mid\, \lambda\text{ is a partition}\}
\]
is a basis for $\text{Sym}(\mathbf{x})$. We will use the notation
\[
[s_{\lambda}]f(\mathbf{x})= \text{the coefficient of $s_{\lambda}$ in the expansion of $f(\mathbf{x})$ in $\mathbf{s}$}. 
\]
We say a symmetric function $f(\mathbf{x})\in \text{Sym}(\mathbf{x})$ is \textit{Schur-positive} if $[s_{\lambda}]f(\mathbf{x})\ge 0$ for all partitions $\lambda$. Some of the reasons for studying Schur-positivity, as well as some of the major results in this area, are discussed in \cite{lam_schur_pos}. 

We now define a family of functions belonging to Sym$(\mathbf{x})$ known as chromatic symmetric functions, which were introduced by Stanley in 1995 \cite{stan_95}. Given a graph $G$ with vertices $V(G)$, a \textit{proper coloring} of $G$ in $q$ colors is a map
\[
\mathcal{C}:V(G)\to\{1,2,3,\dots, q\}
\]
such that, if $u$ and $v$ are adjacent, then $\mathcal{C}(u)\ne \mathcal{C}(v)$. For a proper coloring $\mathcal{C}$ of $G$, we define
    \[
\mathbf{x}^{\mathcal{C}}=x_{\mathcal{C}(v_1)}x_{\mathcal{C}(v_2)}\cdots x_{\mathcal{C}(v_n)}.
    \]
    The \textit{chromatic symmetric function} of a graph $G$ is the formal power series
\[
X_G(\mathbf{x})=\sum_{\mathcal{C}}\mathbf{x}^{\mathcal{C}},
\]
where the sum ranges over all proper colorings $\mathcal{C}$ of $G$.
    The chromatic symmetric function is symmetric since permuting the variables of the function amounts to permuting the colors in each proper coloring. We say that a graph $G$ is \textit{Schur-positive} if $X_G(\mathbf{x})$ is Schur-positive. 

 In order to discuss an important result from \cite{WW2020}, we now introduce a combinatorial object known as a special rim hook tabloid.

 A \textit{rim hook} of \textit{length} $k$ is a sequence of $k$ connected cells in a diagram, each of which lie on the southeast boundary, and whose removal results in a smaller diagram. For any rim hook, we call each path between consecutive cells a \textit{step}. If the cells are in different rows, we use the term \textit{north step} ($N$-step), whereas if the cells are in different columns, we use the term $\textit{east step}$ ($E$-step).

 Let $\kappa=[\kappa_1,\dots,\kappa_{j}]$ be a composition and $\lambda=(\lambda_1,\dots,\lambda_k)$ be a partition. A \textit{rim hook tabloid} of \svw{\emph{shape}} $\lambda$ and \svw{\emph{content}} $\kappa$ is a filling of the cells of the diagram of $\lambda$ with $j$ sequences of connected cells $r_i$ such that $r_1$ is a rim hook of length $\kappa_1$ and, for all $1\le i \le j-1$, if $r_1,\dots, r_i$ are removed from $\lambda$ to form $\tilde{\lambda}$, $r_{i+1}$ is a rim hook of $\tilde{\lambda}$ of length $\kappa_{i+1}$.  A \textit{special rim hook tabloid} (SRH tabloid) is a rim hook tabloid such that every rim hook intersects the first column. We define the \textit{sign} of an SRH tabloid to be
 \[
 \mathrm{sgn}(T)=(-1)^{\text{\# north steps in diagram}}.
 \]
 Moreover, we use the notation $\mathcal{T}_{\lambda}$ to denote the set of all special rim hook tabloids of shape $\lambda$. Given an SRH tabloid $T$, we denote by $\kappa_T$ the \textit{content} of $T$, which is the composition given by the rim hook lengths read from \svw{\emph{the bottom to the top}} of the diagram.

 \begin{example}
Below, we portray all the possible SRH tabloids of shape $\lambda=(4,2,2)$, that is, all elements of the set $\mathcal{T}_{(4,2,2)}$. \svw{Their contents as we go from left to right, and top to bottom, are $(2,2,4)$, $(2,5,1)$, $(3,1,4)$, $(3,5)$, $(6,1,1)$, $(6,2)$.}
 
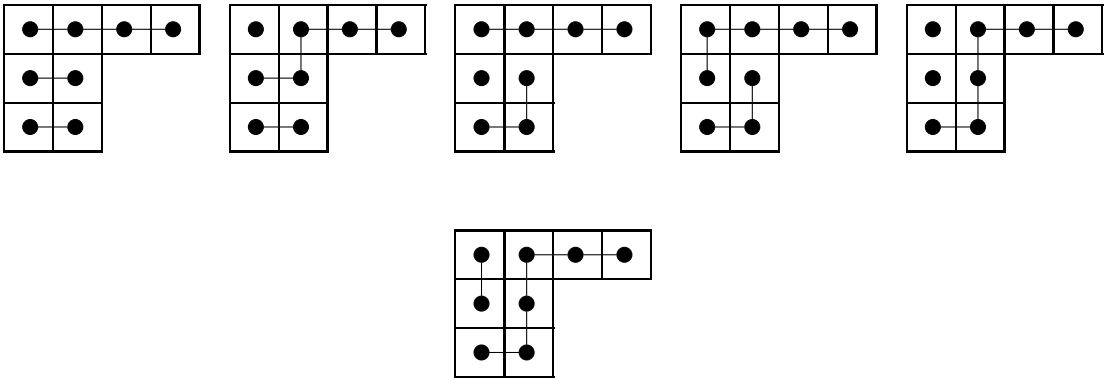
\begin{figure}[H]
 \centering
 \begin{tikzpicture}
 \node at (1,0) {\ydiagram{4,2,2}};
 \node at (.05,.65) [circle, fill = black, scale = .5, draw] (11){};
  \node at (.65,.65) [circle, fill = black, scale = .5, draw] (12){};
  \node at (1.3,.65) [circle, fill = black, scale = .5, draw] (13){};
\node at (1.95,.65) [circle, fill = black, scale = .5, draw] (14){};
\node at (.05,0) [circle, fill = black, scale = .5, draw] (21){};
\node at (.65,0) [circle, fill = black, scale = .5, draw] (22){};
\node at (.05,-.65) [circle, fill = black, scale = .5, draw] (31){};
\node at (.65,-.65) [circle, fill = black, scale = .5, draw] (32){};
\draw (11)--(12)--(13)--(14);
\draw (21)--(22);
\draw (31)--(32);
%
\node at (4,0) {\ydiagram{4,2,2}};
 \node at (3.05,.65) [circle, fill = black, scale = .5, draw] (11){};
  \node at (3.65,.65) [circle, fill = black, scale = .5, draw] (12){};
  \node at (4.3,.65) [circle, fill = black, scale = .5, draw] (13){};
\node at (4.95,.65) [circle, fill = black, scale = .5, draw] (14){};
\node at (3.05,0) [circle, fill = black, scale = .5, draw] (21){};
\node at (3.65,0) [circle, fill = black, scale = .5, draw] (22){};
\node at (3.05,-.65) [circle, fill = black, scale = .5, draw] (31){};
\node at (3.65,-.65) [circle, fill = black, scale = .5, draw] (32){};
\draw (21)--(22)--(12)--(13)--(14);
\draw (31)--(32);
%
%
\node at (7,0) {\ydiagram{4,2,2}};
 \node at (6.05,.65) [circle, fill = black, scale = .5, draw] (11){};
  \node at (6.65,.65) [circle, fill = black, scale = .5, draw] (12){};
  \node at (7.3,.65) [circle, fill = black, scale = .5, draw] (13){};
\node at (7.95,.65) [circle, fill = black, scale = .5, draw] (14){};
\node at (6.05,0) [circle, fill = black, scale = .5, draw] (21){};
\node at (6.65,0) [circle, fill = black, scale = .5, draw] (22){};
\node at (6.05,-.65) [circle, fill = black, scale = .5, draw] (31){};
\node at (6.65,-.65) [circle, fill = black, scale = .5, draw] (32){};
\draw (11)--(12)--(13)--(14);
\draw (31)--(32)--(22);
%
%
\node at (10,0) {\ydiagram{4,2,2}};
 \node at (9.05,.65) [circle, fill = black, scale = .5, draw] (11){};
  \node at (9.65,.65) [circle, fill = black, scale = .5, draw] (12){};
  \node at (10.3,.65) [circle, fill = black, scale = .5, draw] (13){};
\node at (10.95,.65) [circle, fill = black, scale = .5, draw] (14){};
\node at (9.05,0) [circle, fill = black, scale = .5, draw] (21){};
\node at (9.65,0) [circle, fill = black, scale = .5, draw] (22){};
\node at (9.05,-.65) [circle, fill = black, scale = .5, draw] (31){};
\node at (9.65,-.65) [circle, fill = black, scale = .5, draw] (32){};
\draw (21)--(11)--(12)--(13)--(14);
\draw (31)--(32)--(22);
%
%
\node at (13,0) {\ydiagram{4,2,2}};
 \node at (12.05,.65) [circle, fill = black, scale = .5, draw] (11){};
  \node at (12.65,.65) [circle, fill = black, scale = .5, draw] (12){};
  \node at (13.3,.65) [circle, fill = black, scale = .5, draw] (13){};
\node at (13.95,.65) [circle, fill = black, scale = .5, draw] (14){};
\node at (12.05,0) [circle, fill = black, scale = .5, draw] (21){};
\node at (12.65,0) [circle, fill = black, scale = .5, draw] (22){};
\node at (12.05,-.65) [circle, fill = black, scale = .5, draw] (31){};
\node at (12.65,-.65) [circle, fill = black, scale = .5, draw] (32){};
\draw (31)--(32)--(22)--(12)--(13)--(14);
%
%
%
\node at (7,-3) {\ydiagram{4,2,2}};
 \node at (6.05,-2.35) [circle, fill = black, scale = .5, draw] (11){};
  \node at (6.65,-2.35) [circle, fill = black, scale = .5, draw] (12){};
  \node at (7.3,-2.35) [circle, fill = black, scale = .5, draw] (13){};
\node at (7.95,-2.35) [circle, fill = black, scale = .5, draw] (14){};
\node at (6.05,-3) [circle, fill = black, scale = .5, draw] (21){};
\node at (6.65,-3) [circle, fill = black, scale = .5, draw] (22){};
\node at (6.05,-3.65) [circle, fill = black, scale = .5, draw] (31){};
\node at (6.65,-3.65) [circle, fill = black, scale = .5, draw] (32){};
\draw (11)--(21);
\draw (31)--(32)--(22)--(12)--(13)--(14);
 \end{tikzpicture}

 \end{figure}
Counting the parity of $N$-steps, we find the second, third, and sixth tabloids have negative sign whereas the rest of the tabloids have positive sign.

 \end{example}

  Now, let $\lambda$ be a partition and $G$ be a graph. \svw{Recall that if $S$ is a subset of $V(G)$, we say $S$ is a \emph{stable set} if all pairs of vertices in $S$ are
nonadjacent.} A \textit{semi-ordered stable partition} of $G$ is a partition of $V(G)$ into $k$ stable sets (called \textit{parts}) having cardinalities $\lambda_1,\dots,\lambda_k$ such that parts of the same cardinality are ordered.  Considering SRH tabloids and semi-ordered stable partitions allows us to state the following result from \cite{WW2020}.
 
 \begin{theorem}\cite{WW2020}
 \label{thm:WW}
 For any graph $G$ and any partition $\lambda$ of $|V(G)|$, we have
 \begin{align*}
     [s_{\lambda}]X_G=\sum_{T\in \mathcal{T}_{\lambda}}\mathrm{sgn}(T)so_G(T),
 \end{align*}
 where $so_G(T)$ denotes the number of semi-ordered stable partitions of $G$ of type $\Lambda(\kappa_T)$.
 \end{theorem}

 In \cite{WW2020}, the formula is utilized to prove the non-Schur-positivity of subclasses of complete tripartite graphs, squid graphs, and pineapple graphs.

To conclude this section, we state \svwj{the} Claw-free Conjecture, which motivates our main results in Section \ref{sect:gen_nets}. A graph $G$ is \textit{$H$-free} if it does not contain a copy of $H$ as an induced subgraph. The \textit{claw} $K_{13}$, as depicted below, is the smallest non-Schur-positive graph on four vertices.

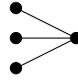
\begin{figure}[H]
    \centering
    \begin{tikzpicture}[scale=.4]

\node[circle,draw,scale=.75,fill=black,scale = .5] at (0,1) (1) {};
\node[circle,draw,scale=.75,fill=black,scale = .5] at (0,0) (2) {};
\node[circle,draw,scale=.75,fill=black,scale = .5] at (0,-1) (3) {};
\node[circle,draw,scale=.75,fill=black,scale = .5] at (2,0) (4) {};
\draw (1)--(4)--(3);
\draw (2)--(4);
    \end{tikzpicture}
\end{figure}
 
\begin{conjecture}[\cite{stan_98}]
\label{conj:claw_free}
 All claw-free graphs are Schur-positive. 
\end{conjecture}
A large class of claw-free graphs are known to be Schur-positive, due to the following theorem.

\begin{theorem}[\cite{gash_96}]
 \label{thm:claw_free_incs}
 All claw-free incomparability graphs are Schur-positive. 
\end{theorem}

\section{Special Rim Hook $G$-Tabloids}
\label{sect:SRH_G_TABS}

  In this section, we introduce an altered version of an SRH tabloid, which corresponds to a graph $G$. 

 \begin{definition}
 \label{def:SRH_G_tabs}
 Consider a graph $G$ and a partial order $\le$ on the vertices of $G$ satisfying: if vertices $u$ and $v$ are nonadjacent, then $u$ and $v$ are comparable.

 We can then define an \textit{SRH $G$-tabloid} to be an SRH tabloid such that every cell is filled with a vertex of $G$ and the following conditions are met.
\begin{itemize}
    \item Cells spanned by the same rim hook contain vertices which form a stable set.
    \item For each rim hook, reading the corresponding vertices from \svwj{the cell in the first column} results in an increasing sequence with respect to the partial order.
\end{itemize}
The \textit{sign} and \textit{shape} of an SRH $G$-tabloid are respectively the sign and shape of the underlying SRH tabloid.
We denote the set of all SRH $G$-tabloids of shape $\lambda$ by $\mathcal{T}_{\lambda,G}$. Note that if the graph $G$ or the partition $\lambda$ is not properly defined, we set $\mathcal{T}_{\lambda,G}=\emptyset$. This convention is to address some notational subtleties appearing in the proofs of Propositions \ref{prop:rec_for_most_coef} and \ref{prop:spi_rec} (Notation \ref{notation:rec_set_up} is also related to this).
 \end{definition}


\begin{definition}
\label{def:tail_head}
Let $T$ be an SRH $G$-tabloid. We define the \textit{tail} of $T$, denoted $\mathrm{tl}(T)$, to be the part of $T$ containing all rows of length $1$. Likewise, we define the \textit{head} of $T$, denoted $\mathrm{hd}(T)$ to be the part of $T$ containing all rows of \svwj{lengths} strictly greater than $1$. Either tl$(T)$ or hd$(T)$ could be empty, depending on the shape of $T$. 
\end{definition}

Let $T$ and $T'$ be SRH $G$-tabloids. We say that
\[
\mathrm{tl}(T)=\mathrm{tl}(T') \orr \mathrm{hd}(T)=\mathrm{hd}(T')
\]
if and only if the tails, or respectively the heads, of $T$ and $T'$ are equal as SRH $G'$-tabloids for some subgraph $G'$ of $G$.
 
 There is some freedom afforded by the choice of a partial order in considering SRH $G$-tabloids. First, if $G$ is the incomparability graph of some poset $P$, it is natural to use the partial order of the underlying poset. Since $G$ is an incomparability graph, two vertices are non-adjacent if and only if they are comparable so the necessary condition is satisfied. If $G$ is not an incomparability graph, we must take a different approach to choosing the partial order. The most direct technique is to choose a total order on $G$ via a numerical labeling of the vertices.

\begin{example}
Consider the graph $G$ depicted below, for which we have chosen a labeling of the vertices. 
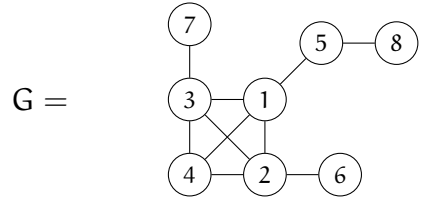
\begin{figure}[H]
\centering
\begin{tikzpicture}
\node[] at (-2,1.5) {$G=$};
\node[circle,draw,scale=.75] at (1,1.5) (1) {$1$};
\node[circle,draw,scale=.75] at (1,.5) (2) {$2$};
\node[circle,draw,scale=.75] at (0,.5) (4) {$4$};
\node[circle,draw,scale=.75] at (0,1.5) (3) {$3$};
\node[circle,draw,scale=.75] at (1.75,2.25) (5) {$5$};
\node[circle,draw,scale=.75] at (2.75,2.25) (8) {$8$};
\node[circle,draw,scale=.75] at (0,2.5) (7) {$7$};
\node[circle,draw,scale=.75] at (2,.5) (6) {$6$};
\draw (8)--(5)--(1)--(2)--(4)--(3)--(2)--(6);
\draw (7)--(3)--(1)--(4);
\end{tikzpicture}
\end{figure}
In this case, $G$ is not an incomparability graph. \svwj{This is because the subgraph obtained by removing the vertices labeled 4 and 8 is not an incomparability graph \cite{stan_95}.} Below, we list some SRH $G$-tabloids of shape $(3,2,1^3),$ that is, elements of $\mathcal{T}_{(3,2,1^3),G},$ using our numerical labeling of the vertices as our partial order. \svw{The bottom three rows form the tail,  the top two rows form the head, and all are different.}

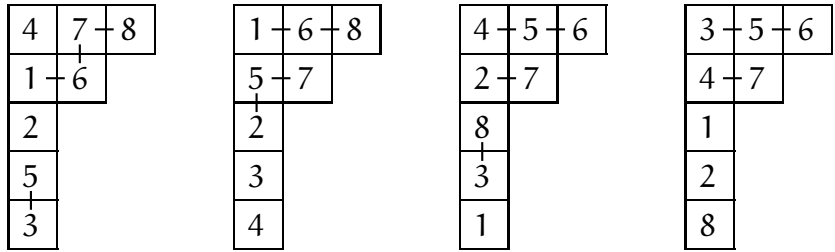
\begin{figure}[H]
\centering
\begin{tikzpicture}
 \node at (1,0) {\ydiagram{3,2,1,1,1}};
  \node at (.325,1.3) [] (4){$4$};
  \node at (.975,1.3) [] (7){$7$};
  \node at (1.625,1.3) [] (8){$8$};
  \node at (.325,.65) [] (1){$1$};
  \node at (.975,.65) [] (6){$6$};
  \node at (.325,0) [] (2){$2$};
  \node at (.325,-.65) [] (5){$5$};
  \node at (.325,-1.3) [] (3){$3$};
  \draw [thick] (.525,.65)--(.775,.65);
  \draw [thick] (.975,.85)--(.975,1.1);
  \draw [thick] (1.175,1.3)--(1.425,1.3);
  \draw [thick] (.325,-1.1)--(.325,-.85);
 \node at (4,0) {\ydiagram{3,2,1,1,1}};
  \node at (3.325,1.3) [] (1){$1$};
  \node at (3.975,1.3) [] (6){$6$};
  \node at (4.625,1.3) [] (8){$8$};
  \node at (3.325,.65) [] (5){$5$};
  \node at (3.975,.65) [] (7){$7$};
  \node at (3.325,0) [] (2){$2$};
  \node at (3.325,-.65) [] (3){$3$};
  \node at (3.325,-1.3) [] (4){$4$};
  \draw [thick] (3.525,.65)--(3.775,.65);
  \draw [thick] (3.525,1.3)--(3.775,1.3);
  \draw [thick] (3.325,.2)--(3.325,.45);
  \draw [thick] (4.175,1.3)--(4.425,1.3);
  \node at (7,0) {\ydiagram{3,2,1,1,1}};
  \node at (6.325,1.3) [] (4){$4$};
  \node at (6.975,1.3) [] (5){$5$};
  \node at (7.625,1.3) [] (6){$6$};
  \node at (6.325,.65) [] (2){$2$};
  \node at (6.975,.65) [] (7){$7$};
  \node at (6.325,0) [] (2){$8$};
  \node at (6.325,-.65) [] (3){$3$};
  \node at (6.325,-1.3) [] (4){$1$};
  \draw [thick] (6.525,.65)--(6.775,.65);
  \draw [thick] (6.525,1.3)--(6.775,1.3);
  \draw [thick] (6.325,-.45)--(6.325,-.2);
  \draw [thick] (7.175,1.3)--(7.425,1.3);
  \node at (10,0) {\ydiagram{3,2,1,1,1}};
  \node at (9.325,1.3) [] (3){$3$};
  \node at (9.975,1.3) [] (5){$5$};
  \node at (10.625,1.3) [] (6){$6$};
  \node at (9.325,.65) [] (4){$4$};
  \node at (9.975,.65) [] (7){$7$};
  \node at (9.325,0) [] (1){$1$};
  \node at (9.325,-.65) [] (2){$2$};
  \node at (9.325,-1.3) [] (8){$8$};
  \draw [thick] (9.525,.65)--(9.775,.65);
  \draw [thick] (9.525,1.3)--(9.775,1.3);
  \draw [thick] (10.175,1.3)--(10.425,1.3);
\end{tikzpicture}
\end{figure}
\end{example}

By considering SRH $G$-tabloids, we may now obtain an alternative combinatorial interpretation of Schur coefficients.

\begin{corollary}
\label{cor:better_s_form}
Consider any graph $G$, a partition $\lambda$ of $|V(G)|$, and a partial order on the vertices of $G$ such that nonadjacent vertices are comparable. We have
\[
[s_{\lambda}]X_G=\sum_{T\in \mathcal{T}_{\lambda,G}}\mathrm{sgn}(T).
\]
\end{corollary}
\begin{proof}
Each summand in the formula of Theorem \ref{thm:WW} counts pairings of an SRH tabloid $T\in \mathcal{T}_{\lambda}$ with semi-ordered stable partitions of $G$ of type $\Lambda(\kappa_T)$. The summand is then assigned a sign equal to $\mathrm{sgn}(T)$. For any $T\in \mathcal{T}_{\lambda}$, we claim the set of pairings of $T$ with semi-ordered stable partitions is in bijection with all SRH $G$-tabloids $T'$ of shape $\lambda$ and content $\kappa_T$. To obtain the map
\begin{align*}
\Psi_T: \{(T,\Omega)\,\mid\, \Omega\text{ is a semi-ordered} &\text{ stable partition of $G$ of type $\Lambda(\kappa_T)$}\}\\
&\longrightarrow \\
\{T'\in\mathcal{T}_{\lambda, G}\,&\mid\, \kappa_{T'}=\kappa_T\},
\end{align*}
we construct an SRH $G$-tabloid from a given $(T,\Omega)$ by identifying each stable set of size $t$ in $\Omega$ with a rim hook of length $t$ in $T$. For any $t\ge 1$, we have the same number of stable sets of size $t$ as rim hooks of length $t$ in $T$ because $\Omega$ is of type $\Lambda(\kappa_T)$. Moreover, the rim hooks of length $t$ in $T$ can be assigned an ordering based on the position (from bottom to top) of their southwest-most cell in the diagram. Since $\Omega$ is also semi-ordered, we can thus uniquely associate the stable sets of size $t$ to rim hooks of length $t$ such that the ordering on the stable sets is consistent with the ordering on the rim hooks.

To obtain an inverse map
\begin{align*}
\Psi_T^{-1}: \{T'\in\mathcal{T}_{\lambda, G}\,&\mid\, \kappa_{T'}=\kappa_T\}\\
\longrightarrow\\
\{(T,\Omega)\,\mid\, \Omega\text{ is a semi-ordered stable}&\text{ partition of $G$ of type $\Lambda(\kappa_T)$}\},
\end{align*}
we let $T$ be the underlying SRH tabloid structure of a given $T'$ and construct a semi-ordered partition $\Omega$ of $G$ by taking parts equal to the sets of vertices associated with each rim hook. We order the parts of equal size according to the order of the corresponding rim hooks by southwest-most cells (as we defined previously). Since $\kappa_{T'}=\kappa_T$, $\Omega$ will be of type
\[
\Lambda(\kappa_{T'})=\Lambda(\kappa_T).
\]
Moreover, this map preserves the signs of the associated tabloids because, by definition, the sign of $T'$ is that of its underlying SRH tabloid $T$. Lastly, we have
\[
\mathcal{T}_{\lambda,G}=\bigsqcup_{T\in\mathcal{T}_{\lambda}}\{T'\in\mathcal{T}_{\lambda, G}\,\mid\, \kappa_{T'}=\kappa_T\}
\]
so constructing bijections $\Psi_T$ for every $T\in \mathcal{T}_{\lambda}$ results in the equality
\[
\sum_{T\in \mathcal{T}_{\lambda}}\mathrm{sgn}(T)so_G(T)=\sum_{T\in \mathcal{T}_{\lambda,G}}\mathrm{sgn}(T).
\]
\end{proof}

\begin{remark}
In the case where $G=\text{inc}(P)$, SRH $G$-tabloids of shape $\lambda$ are in sign-preserving bijection with $P$-arrays of shape $\pi(\lambda)$, as discussed in \cite{gash_96}. Accordingly, Corollary \ref{cor:better_s_form} can also be obtained as a corollary to the proof of Theorem 3 in \cite{gash_96}.
\end{remark}

\section{Generalized Nets}
\label{sect:gen_nets}


\begin{definition}
\label{def:gen_nets}
A \textit{generalized net} $GN_{n,m}$, $n\ge 1$, $n\ge m\ge 0$, is a complete graph on $n$ vertices with $m$ degree one vertices appended to distinct vertices in the complete graph. The \svwj{set of} vertices in the complete graph (of degree $n-1$ and $n$) are collectively referred to as the \textit{body}. The degree $1$ vertices are referred to as \textit{pendants}, the degree $n$ vertices are referred to as \textit{anchors}, and the degree $n-1$ vertices are referred to as \textit{buoys}. 
\end{definition}


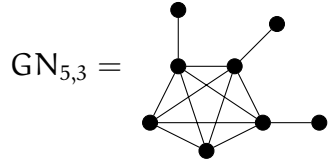
\begin{example}\,
\svwj{The following generalized net has 3 pendants, attached to 3 anchors, and 2 buoys together giving 5 vertices in the body.}
\begin{figure}[H]
\centering
\begin{tikzpicture}[scale=.75]
\node[] at (-2,1.5) {$GN_{5,3}=$};
\node[circle,fill=black, draw,scale = .5] at (1,1.5) (2) {};
\node[circle,fill=black,draw,scale = .5] at (1.5,.5) (3) {};
\node[circle,fill=black,draw,scale = .5] at (-.5,.5) (5) {};
\node[circle,fill=black,draw,scale = .5] at (0,1.5) (1) {};
\node[circle,fill=black,draw,scale = .5] at (.5,0) (4) {};
\node[circle,fill=black,draw,scale = .5] at (1.75,2.25) (7) {};
\node[circle,fill=black,draw,scale = .5] at (0,2.5) (6) {};
\node[circle,fill=black,draw,scale = .5] at (2.5,.5) (8) {};
\draw (1)--(2)--(3)--(4)--(5)--(1)--(3)--(5)--(2)--(4)--(1);
\draw (1)--(6);
\draw (2)--(7);
\draw (3)--(8);
\end{tikzpicture}
\label{fig:net_example}
\end{figure}
\end{example}

One may observe directly that all generalized nets are claw-free. Generalized nets are also not incomparability graphs whenever there are three or more \svwj{pendants; one way} to see this is that these graphs contain the subgraph $GN_{3,3}$, which is used in \cite{stan_95} as one of the simplest examples of a claw-free non-incomparability graph.

Throughout this section, we work with two different choices of labelings on generalized nets, each of which is advantageous in distinct situations.

\begin{definition}
In a \textit{pendant-first labeling} for a generalized net, we label the pendant vertices $1,\dots,m,$ we label the anchors $m+1,\dots,2m$ \svwj{so that each anchor $m+i$ is adjacent to the pendant labeled $i$ for $1\le i \le m$}, and we label the buoys $2m+1,\dots, n+m$. In a \textit{pendant-last labeling} for a generalized net, we label the buoys $1,\dots,n-m,$ we label the anchors $n-m+1,\dots,n,$ and we label the pendants $n+1,\dots,n+m$ (so that each anchor $i$ is adjacent to the pendant labeled $i+m$ for $n-m+1\le i \le n$).
\end{definition}

\begin{example}

Below, we depict a pendant-first labeling of $GN_{4,2}$ on the left and a pendant-last labeling of $GN_{4,2}$ on the right. \svw{On the left, vertices 1 and 2 are pendants, 3 and 4 are anchors, and 5 and 6 are buoys. Meanwhile, on the right, vertices 1 and 2 are buoys, 3 and 4 are anchors, and 5 and 6 are pendants.}
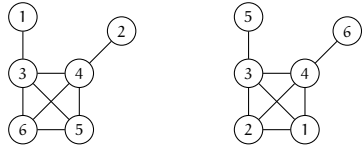
\begin{figure}[H]
\centering
\begin{tikzpicture}[scale = .75]
\node[circle,draw,scale = .5] at (1,1.5) (4) {$4$};
\node[circle,draw,scale = .5] at (1,.5) (5) {$5$};
\node[circle,draw,scale = .5] at (0,.5) (6) {$6$};
\node[circle,draw,scale = .5] at (0,1.5) (3) {$3$};
\node[circle,draw,scale = .5] at (1.75,2.25) (2) {$2$};
\node[circle,draw,scale = .5] at (0,2.5) (1) {$1$};
\draw (1)--(3)--(4)--(5)--(6)--(3)--(5);
\draw (6)--(4)--(2);
\node[circle,draw,scale = .5] at (5,1.5) (4) {$4$};
\node[circle,draw,scale = .5] at (5,.5) (5) {$1$};
\node[circle,draw,scale = .5] at (4,.5) (6) {$2$};
\node[circle,draw,scale = .5] at (4,1.5) (3) {$3$};
\node[circle,draw,scale = .5] at (5.75,2.25) (2) {$6$};
\node[circle,draw,scale = .5] at (4,2.5) (1) {$5$};
\draw (1)--(3)--(4)--(5)--(6)--(3)--(5);
\draw (6)--(4)--(2);
\end{tikzpicture}
\label{fig:pend_first_vs_pend_last}
\end{figure}
\end{example}




We start with a lemma concerning the number of pendants which may be in the tail of an  SRH $GN_{n,m}$-tabloid.

\begin{lemma}
\label{lemma:tail_cant_be_pend_filled}
Consider a partition $\lambda=(\lambda_1,\dots,\lambda_k)$ \svwj{with} $\lambda_k=1$.  For any $T\in\mathcal{T}_{\lambda, GN_{n,m}}$, $n\ge 1$, $n\ge m\ge 0$, we have that, regardless of the choice of labeling, the tail cannot contain only pendants.
\end{lemma}
\begin{proof}
Assume the tail of $T$ has $h$ pendants and no vertices from the body. The head must contain $n+m-h$ cells to include all the other vertices. Accordingly, the head may have at most $\lfloor \frac{n+m-h}{2} \rfloor$ rows and at most $\lfloor \frac{n+m-h}{2} \rfloor$ distinct rim hooks since every rim hook must intersect the first column. We then have
\[
\#\text{vertices from the body in head}\le \# \text{rim hooks in the head} \le \left\lfloor \frac{n+m-h}{2} \right\rfloor < n.
\]
The first inequality follows since every \svw{body} vertex \svw{in} the head must be in its own rim hook \svw{by definition}. The last inequality follows since $m\le n$ and $h\ge 1$. Hence, at least one of the $n$ vertices from the body must be in the tail, contradicting the initial assumption.
\end{proof}

The next proposition shows that certain sets of SRH $G$-tabloids can be canceled out in terms of sign via an algorithm involving rearranging the pendants in the tails of these tabloids. We note that the terms ``pendants'' and ``body vertices'' are used in this proposition in a  context \svwj{that is more general} than in the definition of generalized nets, \svwj{and in the former case is less general than the classic definition of ``pendant'' that refers to any vertex of degree 1.}



\begin{proposition}
    \label{prop:rec_algorithm}
    Let $G$ be a graph, let $\lambda=(\lambda_1,\dots,\lambda_k)$ be a partition such that $\lambda_k=\lambda_{k-1}=1$, and let $\mathcal{S}\subseteq \mathcal{T}_{\lambda,G}$ be a subset of SRH $G$-tabloids such that
\[\mathrm{hd}(T)=\mathrm{hd}(T')=H\text{ for all }T, T'\in\mathcal{S}.\]
Let $\mathcal{V}$ denote the set of vertices of $G$ appearing in $\mathrm{tl}(T)$ for all $T\in\mathcal{S}$, and suppose
\[\mathcal{V}=\mathcal{P}\sqcup \mathcal{U},\]
where $\mathcal{P}$ and $\mathcal{U}$ satisfy the following conditions.
\begin{enumerate}[(I)]
    \item We have that $\mathcal{P}$ is a stable set of vertices, called pendants, which are degree $1$ in $G$ and adjacent to at most one vertex in $\mathcal{U}$. If $p\in \mathcal{P}$ is adjacent to some $u\in \mathcal{U}$, we refer to $u$ as the anchor corresponding to  \svw{$p$.}
    \item We have that $\mathcal{U}$ is a nonempty set of vertices, called body vertices, such that each $u\in \mathcal{U}$ is adjacent to at most one pendant.
\end{enumerate}
Lastly, suppose that $\mathcal{S}$ contains exactly the tabloids $T\in\mathcal{T}_{\lambda,G}$ for which \[\mathrm{hd} (T)=H,\] and for which the following conditions are satisfied.
\begin{enumerate}[(i)]
    \item  All distinct body vertices $u,u'\in \mathcal{U}$ are in different rim hooks. 
    \item The bottom cell in $T$ is filled by a pendant $p\in \mathcal{P}$ which is nonadjacent to the vertex in the cell above.
\end{enumerate}
We then have
\[
\sum_{T\in\mathcal{S}}\mathrm{sgn}(T)=0.
\]
\end{proposition}

\begin{proof}
We label the vertices of $G$ such that all pendants have labels smaller than the labels of the body vertices.

Consider all $T\in\mathcal{S}$ such that an $N$-step up from the \svwj{bottom} pendant is permissible. We can define a sign-reversing involution on these tabloids by adding the $N$-step if it is not there and removing the $N$-step if it is there.

Thus, we are left counting tabloids for which no $N$-step up from the \svwj{bottom} vertex is allowed. These tabloids fall under two cases.
\begin{enumerate}[\text{Case} (1),leftmargin=*]
    \item The pendant $p$ in the bottom row is below a pendant $q$ such that $p>q$.
    \item The pendant $p$ in the bottom row is below a sequence of pendants $y_1,\dots,y_{d},$ ending with the anchor $u$ corresponding to $p$ such that
    \[
    p<y_1<\cdots<y_{d}<u,
    \]
    $|\mathcal{P}|-1 \ge d\ge 1$, and all consecutive vertices are connected by $N$-steps except $p$ and $y_1$.
\end{enumerate}
In any other case, an edge up from $p$ is permissible. Recall that Condition (i) is that no two body vertices are in the same rim hook.
\svwj{Moreover,} Condition (ii) ensures $p$ is not directly below its corresponding anchor $u$.

Consider any $T$ which falls under Case $(2)$. We map $T$ to the otherwise identical tabloid with the new sequence \[y_1>p<y_2<\cdots<y_d<u.\] This new tabloid falls under Case $(1)$. Moreover, since one $N$-step (from $y_1$ to $y_2$) is removed, this map is sign-reversing.

After applying this map, we consider the remaining Case $(1)$ tabloids. If an $N$-step up from $q$ is permissible, we apply a map which either adds or removes this $N$-step. Hence, we are left again with two subcases, for which no $N$-step up from $q$ is permissible.

\begin{enumerate}[\text{Case} (a),leftmargin=*]
    \item Reading up from the bottom rows, we have the sequence $p>q>r$ for some pendant $r$.
    \item Reading up from the bottom rows, we have the sequence ending with the anchor $u'$ corresponding to $q$ such that
    \[p>q<y_1<\cdots< y_{d'}<u',\]
    where $|\mathcal{P}|-2\ge d'\ge 1$, all consecutive vertices from $y_1$ and on are connected by $N$-steps, and $p>y_1>q$. We note that $d'\ge 1$ since the tabloids with sequence $p>q<u'$ are canceled out by tabloids with sequence $q<p<u'$ in Case $(2)$. Likewise, the condition $p>y_1>q$ holds since tabloids with $p<y_1$ are canceled by the Case $(2)$ tabloids with sequence
    \[
    q<p<y_1<\cdots<y_{d'}<u.
    \]
\end{enumerate}

Take any $T$ which falls under Case (b). We map this tabloid to the otherwise identical tabloid with the new sequence
\[
p>y_1>q<\cdots< y_{d'}<u'.
\]
Once again, this map removes one $N$-step (from $y_1$ to $y_2$) so it is sign-reversing. Moreover, the new tabloid falls under Case (a).

We then cancel out all Case (a) tabloids such that an $N$-step up from $r$ is permissible by adding or removing that $N$-step. We proceed by applying the same method iteratively to the leftover tabloids which fall under Case (a).

On the $j$th step of this iteration for $2\le j \le |\mathcal{P}|-1$, we have tabloids of two cases.
\begin{enumerate}[\text{Case} (A),leftmargin=*]
    \item Reading up from the bottom rows, we have the sequence of pendants
    \[
    p_1>p_2>\cdots > p_{j+1}.
    \]
    \item Reading up from the bottom rows, we have the sequence ending with the anchor $u''$ corresponding to $p_j$ such that
    \[
    p_1>p_2>\cdots > p_j<y_1<\cdots <y_{d''}<u'',
    \]
    where $|\mathcal{P}|-j\ge d''\ge 1$, all consecutive vertices from $y_1$ and on are connected by $N$-steps, and $p_{j-1}>y_1>p_j$. As before, these conditions arise from the other tabloids being canceled out by the map in the $j-1$th step.
\end{enumerate}
We take any tabloid covered by Case (B) and map it to the otherwise identical tabloid with the new sequence
\[
 p_1>p_2>\cdots >p_{j-1}>y_1>p_j<\cdots <y_{d''}<\svw{u'',}
\]
which removes one $N$-step (from $y_1$ to $y_2$) and thus is sign-reversing. We then apply a sign-reversing involution to all Case (A) tabloids for which an $N$-step up from $p_{j+1}$ is permissible (by adding or removing that $N$-step). We then proceed with the $j+1$th step to cancel out the remaining tabloids.

When $j=|\mathcal{P}|-1$, we consider tabloids with sequences (from the bottom up) ending with the anchor $u'''$ corresponding to $p_{|\mathcal{P}|-1}$, such that
\[
p_1>p_2>\cdots>p_{|\mathcal{P}|-2} >p_{|\mathcal{P}|-1}<y_1<u''',
\]
where $y_1$ is connected to $u'''$ by an $N$-step and $p_{|\mathcal{P}|-2}>y_1>p_{|\mathcal{P}|-1}$. We cancel these tabloids out with the otherwise identical tabloids with the new sequence
\[
p_1>p_2>\cdots p_{|\mathcal{P}|-2}>y_1>p_{\svwj{|\mathcal{P}|-1}}<u'''.
\]
Accordingly, the only tabloids which remain have the sequence
\[
p_1>p_2>\cdots > p_{|\mathcal{P}|},
\]
read from the bottom up in their tails. For each of these tabloids, there are no more pendants in the tail besides in this sequence. Furthermore, the tabloids with the anchor $u'''$ above $p_{|\mathcal{P}|}$ have been canceled and $\mathcal{U}\ne \emptyset$ so some other $u\in \mathcal{U}$ is above $p_{|\mathcal{P}|}$. Thus, an $N$-step up from $p_{|\mathcal{P}|}$ is permissible for all remaining tabloids. Hence, we cancel these out via a sign-reversing involution which adds or removes that $N$-step.

We note that at no point in this algorithm do we apply a map which results in two body vertices being in the same rim hook, so Condition (i) is satisfied for all the tabloids we considered. 
Moreover, all the tabloids we considered have pendants in the bottom row, nonadjacent to the vertex in the cell above, so Condition (ii) is always satisfied as well. Since $\mathcal{S}$ contains all tabloids satisfying the conditions of the proposition, all the maps send elements of $\mathcal{S}$ to other elements of $\mathcal{S}$.

We conclude
\[
\sum_{T\in\mathcal{S}}\mathrm{sgn}(T)=0.
\]
\end{proof}


\begin{example}
Consider some $\mathcal{S}\subseteq \mathcal{T}_{\lambda, G}$ satisfying the conditions of Proposition \ref{prop:rec_algorithm}, where
\[
\lambda=(2^2,1^{8}),\qquad G=GN_{6,6},\qquad |\mathcal{P}|=4,\andd |\mathcal{U}|=4.\]

In this example, we depict sign-reversing maps on tabloids in $\mathcal{S}$, as used in the second step of the iteration in the proof of Proposition \ref{prop:rec_algorithm}, in the case where $d'=2$. First, we demonstrate below how the tabloids which fall under Case (b) are canceled.


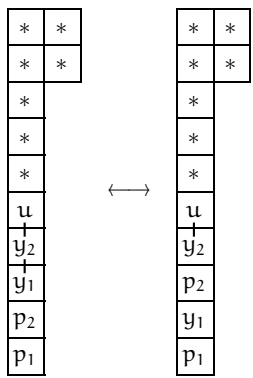
\begin{figure}[H]
\centering
\begin{tikzpicture}[scale=.75]
 \node[scale=.75] at (1,0) {\ydiagram{2,2,1,1,1,1,1,1,1,1}};
 \node[scale=.75] at (.65,2.9) [] (a){$*$};
 \node[scale=.75] at (1.3,2.9) [] (a){$*$};
 \node[scale=.75] at (.65,2.25) [] (a){$*$};
   \node[scale=.75] at (.65,1.6) [] (a){$*$}
  ;
   \node[scale=.75] at (1.3,2.25) [] (c){$*$};
   \node[scale=.75] at (.65,.95) [] (d){$*$};
   \node[scale=.75] at (.65,.3) [] (f){$*$};
   \node[scale=.75] at (.65,-.35) [] (e){$u$};
   \node[scale=.75] at (.65,-1) [] (i){$y_2$};
   \node[scale=.75] at (.65,-1.65) [] (i){$y_1$};
   \node[scale=.75] at (.65,-2.3) [] (i){$p_2$};
   \node[scale=.75] at (.65,-2.95) [] (i){$p_1$};
  \draw [thick] (.65,-.8)--(.65,-.55);
  \draw [thick] (.65,-1.45)--(.65,-1.2);
 \node[scale=.75] at (2.5,0) {$\longleftrightarrow$};
  \node[scale=.75] at (4,0) {\ydiagram{2,2,1,1,1,1,1,1,1,1}};
 \node[scale=.75] at (3.65,2.9) [] (a){$*$};
 \node[scale=.75] at (4.3,2.25) [] (a){$*$};
 \node[scale=.75] at (4.3,2.9) [] (a){$*$};
 \node[scale=.75] at (3.65,2.25) [] (a){$*$};
   \node[scale=.75] at (3.65,1.6) [] (a){$*$};
   \node[scale=.75] at (3.65,.95) [] (d){$*$};
   \node[scale=.75] at (3.65,.3) [] (f){$*$};
   \node[scale=.75] at (3.65,-.35) [] (e){$u$};
   \node[scale=.75] at (3.65,-1) [] (i){$y_2$};
   \node[scale=.75] at (3.65,-1.65) [] (i){$p_2$};
   \node[scale=.75] at (3.65,-2.3) [] (i){$y_1$};
   \node[scale=.75] at (3.65,-2.95) [] (i){$p_1$};
  \draw [thick] (3.65,-.8)--(3.65,-.55);
\end{tikzpicture}
\label{fig:tab_cancelation1}
\end{figure}
In the above depiction, $u$ is the anchor corresponding to $p_2$, which is why an $N$-step up from $p_2$ is not permitted. We also have the relation $p_1>y_1>p_2$ (the cases where $y_1>p_1>p_2$ were canceled by the first step of the iteration). The tabloid on the left falls under Case (b) and the tabloid on the right falls under Case (a).

We illustrate below how we cancel out tabloids under Case (a) for which an $N$-step up from $p_3$ is permissible. 

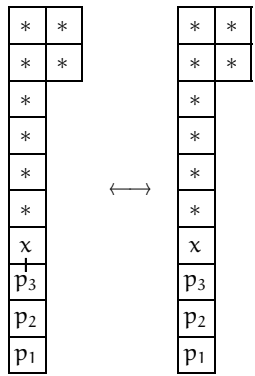
\begin{figure}[H]
\centering
\begin{tikzpicture}[scale=.75]
  \node[scale=.75] at (1,0) {\ydiagram{2,2,1,1,1,1,1,1,1,1}};
 \node[scale=.75] at (.65,2.9) [] (a){$*$};
 \node[scale=.75] at (1.3,2.9) [] (a){$*$};
 \node[scale=.75] at (.65,2.25) [] (a){$*$};
   \node[scale=.75] at (.65,1.6) [] (a){$*$}
  ;
   \node[scale=.75] at (1.3,2.25) [] (c){$*$};
   \node[scale=.75] at (.65,.95) [] (d){$*$};
   \node[scale=.75] at (.65,.3) [] (f){$*$};
   \node[scale=.75] at (.65,-.35) [] (e){$*$};
   \node[scale=.75] at (.65,-1) [] (i){$x$};
   \node[scale=.75] at (.65,-1.65) [] (i){$p_3$};
   \node[scale=.75] at (.65,-2.3) [] (i){$p_2$};
   \node[scale=.75] at (.65,-2.95) [] (i){$p_1$};
  \draw [thick] (.65,-1.45)--(.65,-1.2);
 \node[scale=.75] at (2.5,0) {$\longleftrightarrow$};
 \node[scale=.75] at (4,0) {\ydiagram{2,2,1,1,1,1,1,1,1,1}};
 \node[scale=.75] at (3.65,2.9) [] (a){$*$};
 \node[scale=.75] at (4.3,2.25) [] (a){$*$};
 \node[scale=.75] at (4.3,2.9) [] (a){$*$};
 \node[scale=.75] at (3.65,2.25) [] (a){$*$};
   \node[scale=.75] at (3.65,1.6) [] (a){$*$};
   \node[scale=.75] at (3.65,.95) [] (d){$*$};
   \node[scale=.75] at (3.65,.3) [] (f){$*$};
   \node[scale=.75] at (3.65,-.35) [] (e){$*$};
   \node[scale=.75] at (3.65,-1) [] (i){$x$};
   \node[scale=.75] at (3.65,-1.65) [] (i){$p_3$};
   \node[scale=.75] at (3.65,-2.3) [] (i){$p_2$};
   \node[scale=.75] at (3.65,-2.95) [] (i){$p_1$};
\end{tikzpicture}
\label{fig:tab_cancelation2}
\end{figure}

In this case, $p_1>p_2>p_3$. One possibility is that $x$ is a pendant $p_4>p_3$ which is in a rim hook that does not include the anchor $u'$ corresponding to $p_3$. Otherwise, $x$ is a body vertex other than $u'$.

After applying these maps, we are left with only Case (a) tabloids for which an $N$-step up from $p_3$ is not permissible. These tabloids fall under the two subcases depicted below. 


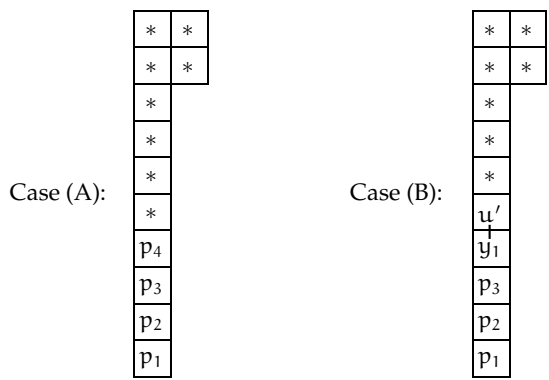
\begin{figure}[H]
\centering
\begin{tikzpicture}[scale=.75]
  \node[scale=.75] at (-4,0) {$\text{Case (A):}$};
  \node[scale=.75] at (-2,0) {\ydiagram{2,2,1,1,1,1,1,1,1,1}};
\node[scale=.75] at (-2.35,2.9) [] (a){$*$};
\node[scale=.75] at (-1.7,2.9) [] (a){$*$};
\node[scale=.75] at (-2.35,2.25) [] (a){$*$};
  \node[scale=.75] at (-2.35,1.6) [] (a){$*$}
  ;
  \node[scale=.75] at (-1.7,2.25) [] (c){$*$};
  \node[scale=.75] at (-2.35,.95) [] (d){$*$};
  \node[scale=.75] at (-2.35,.3) [] (f){$*$};
  \node[scale=.75] at (-2.35,-.35) [] (e){$*$};
  \node[scale=.75] at (-2.35,-1) [] (i){$p_4$};
  \node[scale=.75] at (-2.35,-1.65) [] (i){$p_3$};
  \node[scale=.75] at (-2.35,-2.3) [] (i){$p_2$};
  \node[scale=.75] at (-2.35,-2.95) [] (i){$p_1$};
 \node[scale=.75] at (2,0) {$\text{Case (B):}$};
 \node[scale=.75] at (4,0) {\ydiagram{2,2,1,1,1,1,1,1,1,1}};
\node[scale=.75] at (3.65,2.9) [] (a){$*$};
\node[scale=.75] at (4.3,2.25) [] (a){$*$};
\node[scale=.75] at (4.3,2.9) [] (a){$*$};
\node[scale=.75] at (3.65,2.25) [] (a){$*$};
  \node[scale=.75] at (3.65,1.6) [] (a){$*$};
  \node[scale=.75] at (3.65,.95) [] (d){$*$};
  \node[scale=.75] at (3.65,.3) [] (f){$*$};
  \node[scale=.75] at (3.65,-.35) [] (e){$u'$};
  \node[scale=.75] at (3.65,-1) [] (i){$y_1$};
  \node[scale=.75] at (3.65,-1.65) [] (i){$p_3$};
  \node[scale=.75] at (3.65,-2.3) [] (i){$p_2$};
  \node[scale=.75] at (3.65,-2.95) [] (i){$p_1$};
  \draw [thick] (3.65,-.8)--(3.65,-.55);
\end{tikzpicture}
\label{fig:tab_cancelation3}
\end{figure}
In Case (A), we have that $p_1>p_2>p_3>p_4$, and in Case (B), we have that $p_1>p_2>y_1>p_3$. 
Since $|\mathcal{P}|=4$, the next step of the algorithm will cancel all remaining tabloids out.
\end{example}

We now introduce some helpful notation.
\begin{notation}\hfill
\label{notation:rec_set_up}
\begin{enumerate}
    \item \svwj{For any partition $\lambda=(\lambda_1,\dots,\lambda_k)$ with $\lambda_k=1$,} we denote the subset of $\mathcal{T}_{\lambda,G}$ which has vertex $v$ in the \svwj{bottom} cell of the first column by $\mathcal{T}_{\lambda,G}^v$. Likewise, for some subset $S$ of vertices of $G$, we use the notation
\[
\mathcal{T}_{\lambda,G}^S=\bigcup_{v\in S}\mathcal{T}_{\lambda,G}^v.
\]
\item Next, we define a function that outputs Schur coefficients of a graph's chromatic symmetric function but vanishes whenever either the graph or the coefficient is not properly defined. Explicitly, we let
\[
\xi(\lambda,G)=\begin{cases} [s_{\lambda}]X_G&\qquad\text{if $G$ is a properly defined graph and $\lambda$ is a}\\
&\qquad\text{properly defined partition,}\\
0 &\qquad\text{otherwise}.
\end{cases}
\]
We note that $GN_{n,m}$ where $m>n$ is an example of a graph that is not properly defined. 
\item We also define an operation on partitions corresponding to tabloids with tails of length at least $1$.

Given a partition $\lambda=(\lambda_1,\dots,\lambda_k)$ such that $\lambda_{h},\dots,\lambda_k=1$ for $1\le h \le k$, we define
\[
\lambda\setminus 1^{t}=(\lambda_1,\dots,\lambda_{k-t})
\]
for $t\le k-h+1$. We note if we consider $\lambda\setminus 1^t$ in the case where $\lambda$ includes strictly fewer than $t$ integers equal to $1$, we end up with an undefined partition. 
\end{enumerate}
\end{notation}





\begin{proposition}
\label{prop:rec_for_most_coef}
Let $\lambda=(\lambda_1,\dots,\lambda_k)$ be a partition and assume $\lambda_k=1$. We have
\begin{align}
\begin{split}
\label{eq:rec_for_most_coef}
\xi(\lambda,GN_{n,m})&=m\xi(\lambda\setminus 1,GN_{n-1,m-1}\cup P_1)+(n-m)\xi(\lambda\setminus 1,GN_{n-1,m})\\&+m\xi(\lambda\setminus 1^2,GN_{n-1,m-1})
\end{split}
\end{align}
for $n\ge m\ge 1$.
\end{proposition}
\begin{proof}
We denote the set of all anchors of $GN_{n,m}$ by $A$, the set of all buoys of $GN_{n,m}$ by $B$, and the set of all pendants of $GN_{n,m}$ by $P$. We count tabloids with a pendant-first labeling. We observe that
\begin{align}
\label{eq:reg_rec_tab_partition}
\mathcal{T}_{\lambda, GN_{n,m}}=\mathcal{T}_{\lambda,GN_{n,m}}^A\sqcup \mathcal{T}_{\lambda,GN_{n,m}}^B\sqcup \mathcal{T}_{\lambda,GN_{n,m}}^P.
\end{align}

In the case where an anchor $a\in A$ or buoy $b\in B$ is in the bottom row, there is no $N$-step up from the cell since we are using a pendant-first labeling. Hence, we can map these tabloids to tabloids for which the bottom cell is removed. These are sign-preserving bijections
\begin{align}
\label{eq:reg_rec_biject}
\mathcal{T}_{\lambda,GN_{n,m}}^a\cong \mathcal{T}_{\lambda\setminus 1 ,GN_{n-1,m-1}\cup P_1}\andd \mathcal{T}_{\lambda,GN_{n,m}}^b\cong \mathcal{T}_{\lambda\setminus 1,GN_{n-1,m}},
\end{align}
since removing an anchor detaches a pendant, decreasing the number of body vertices and the number of pendants by one, and removing a buoy just decreases the number of body vertices by one.

From Equation \ref{eq:reg_rec_tab_partition} and Equation \ref{eq:reg_rec_biject}, we then have
\begin{align}
\begin{split}
\label{eq:reg_rec_inc_form}
\xi(\lambda,GN_{n,m})&=m\cdot\sum_{T\in \mathcal{T}_{\lambda\setminus 1 ,GN_{n-1,m-1}\cup P_1}}\mathrm{sgn}(T)+(n-m)\cdot\sum_{T\in \mathcal{T}_{\lambda\setminus 1,GN_{n-1,m}}}\mathrm{sgn}(T)\\
&+\sum_{T\in \mathcal{T}_{\lambda,GN_{n,m}}^P}\mathrm{sgn}(T)\\
&=m\xi(\lambda\setminus 1,GN_{n-1,m-1}\cup P_1)+(n-m)\xi(\lambda\setminus 1,GN_{n-1,m})\\
&+\sum_{T\in \mathcal{T}_{\lambda,GN_{n,m}}^P}\mathrm{sgn}(T)
\end{split}
\end{align}
since there are $m$ anchors and $n-m$ buoys in the graph.


If $\lambda_{k-1}\ne 1$, there are no pendants in the tail of any tabloid $T\in\mathcal{T}_{\lambda,GN_{n,m}}$ by Lemma \ref{lemma:tail_cant_be_pend_filled}. Therefore, we have
\[
\mathcal{T}^P_{\lambda, GN_{n,m}}=\emptyset\so \sum_{T\in \mathcal{T}_{\lambda,GN_{n,m}}^P}\mathrm{sgn}(T)=0.
\]
Moreover, if $\lambda_{k-1}\ne 1$,
\[
\xi(\lambda\setminus 1^2,GN_{n-1,m-1})=0
\]
also holds. Thus, in this case, Equation \ref{eq:rec_for_most_coef} holds.

For the remainder of the proof, we assume $\lambda_{k-1}=1$.

Consider the subset of tabloids in  $\mathcal{T}^P_{\lambda, GN_{n,m}}$ such that a certain pendant $p$ is in the bottom cell and is directly below its anchor. Since there are no $N$-steps up from the anchor, we map these tabloids to tabloids where the bottom two cells are removed. We can accordingly obtain sign-preserving bijections from the aforementioned subset to
\[
\mathcal{T}_{\lambda\setminus 1^2,GN_{n-1,m-1}},
\]
for each pendant $p$,
since removing a pendant and corresponding anchor lowers the number of pendants and the number of body vertices by one, \svwj{respectively.}

We can thus count these subsets of tabloids by adding the term
\[
m\sum_{T\in\mathcal{T}_{\lambda\setminus 1^2,GN_{n-1,m-1}}}\mathrm{sgn}(T)=m\xi(\lambda\setminus 1^2,GN_{n-1,m-1})
\]
to the sum we are computing (we have a factor of $m$ since there are $m$ pendants).

Denote the set of the remaining tabloids in $\mathcal{T}^P_{\lambda, GN_{n,m}}$ by $\mathcal{S}$. We partition $\mathcal{S}$ such that
\[
\mathcal{S}=\bigsqcup_{H}\mathcal{S}_H,
\]
where the disjoint union spans over all possible heads of tabloids $T\in \mathcal{S}$ and each $\mathcal{S}_H\subseteq \mathcal{S}$ is exactly the subset \svwj{of taboids with head $H$, namely}
\[
\mathrm{hd}(T)=\mathrm{hd}(T')=H
\]
for all $T,T'\in\mathcal{S}_H$.

For any $\mathcal{S}_H$, let $\mathcal{P}_H$ and $\mathcal{U}_H$ respectively denote the set of pendants and the set of body vertices appearing in the tail of every $T\in\mathcal{S}_H$. By Lemma \ref{lemma:tail_cant_be_pend_filled} and the definition of a generalized net, Conditions (I) and (II) of Proposition \ref{prop:rec_algorithm} are satisfied for $\mathcal{P}_H$ and $\mathcal{U}_H$.

Moreover, all $u,u'\in U$ are all adjacent so Condition (i) of Proposition \ref{prop:rec_algorithm} is also satisfied. Lastly, recall $\mathcal{S}$ does not include any tabloids with a pendant in the bottom cell directly below its anchor. Therefore, Condition (ii) of Proposition \ref{prop:rec_algorithm} holds as well. We conclude
\[
\sum_{T\in\mathcal{S}_H}\mathrm{sgn}(T)=0
\qquad\text{for all $H$ so}\qquad
\sum_{T\in\mathcal{S}}\mathrm{sgn}(T)=0.
\]
\end{proof}


We now address the case where coefficients correspond to partitions which do not end in $1$.

\begin{lemma}
Assume $[s_{\lambda}]X_{GN_{n,m}}\ne 0$ for some $\lambda=(\lambda_1,\dots,\lambda_k)$ such that $\lambda_k\ne 1$.
Then $n=m$ and $\lambda=(2^n)$.
\label{lemma:only_tailless}
\end{lemma}
\begin{proof}
Let $GN_{n,m}$ have a pendant-last labeling. There may be at most one vertex from the body in each rim hook in $T$. Moreover, the vertices from the body must all be positioned at the beginning of their rim hooks since their labels are smaller than that of the pendants. Since every rim hook must intersect the first column, this implies all $n$ anchors and buoys must be in the first column.

In order for the tail to have no cells, every row in $T$ must have at least $2$ cells. Hence, the second column of $T$ must also have $n$ cells. This implies we must have $n$ pendants to fill the second column. There are no more vertices in the graph so $n=m$ and $\lambda=(2^n)$ must both hold.
\end{proof}


\begin{lemma}
\label{lemma:get_rid_of_singleton}
We have that
\[
[s_{(2^C,1^D)}]X_{GN_{C+D,C-1}\cup P_1}=[s_{(2^{C-1},1^{D+1})}]X_{GN_{C+D,C-1}}
\]
for $C\ge 1, D\ge 0$.
\end{lemma}
\begin{proof}
We label $GN_{C+D,C-1}\cup P_1$ with a pendant-last labeling and assign the degree $0$ vertex (which we call $x$) the largest label. For any $T\in\mathcal{T}_{\lambda,GN_{C+D,C-1}\cup P_1}$, we have that the first column must be completely filled with the $C+D$ vertices from the body (since these vertices are each in distinct rim hooks and labeled minimally). Moreover, this implies there are no $N$-steps in the first column.

Consider any tabloid for which $x$ is not in the \svwj{bottom} cell in the second column. Since $x$ is nonadjacent to every other vertex and labeled maximally, an $N$-step from the vertex below $x$ and an $E$-step from the vertex to the left of $x$ are both permissible.

Accordingly, we map all tabloids with an $N$-step to $x$ to the otherwise identical tabloids with an $E$-step to $x$ and vice versa. This map is a sign-reversing involution (on the tabloids for which $x$ is not in the \svwj{bottom} cell in the second column) since an $N$-step is always added or removed.

Thus, it remains to count the tabloids which have $x$ in the \svwj{bottom} cell in the second column. Since $x$ is labeled maximally, there is no $N$-step up from $x$ so there must be an $E$-step to $x$ from the vertex to the left of $x$.

Consider the map which sends these tabloids to tabloids in $\mathcal{T}_{(2^{C-1},\svwj{1^{D+1}}),GN_{C+D,C-1}}$ (under a pendant-last labeling) by removing the cell containing $x$. Since no $N$-steps are added or removed, and since every $T\in \mathcal{T}_{(2^{C-1},\svwj{1^{D+1}}),GN_{C+D,C-1}}$ is in the image of the map, this is a sign-preserving bijection. We conclude
\[
[s_{(2^C,1^D)}]X_{GN_{C+D,C-1}\cup P_1}=[s_{(2^{C-1},1^{D+1})}]X_{GN_{C+D,C-1}}.
\]
\end{proof}

\begin{example}
Below, we illustrate examples of the two bijections in the proof of Lemma \ref{lemma:get_rid_of_singleton}, in the case where $C=3,D=3$.
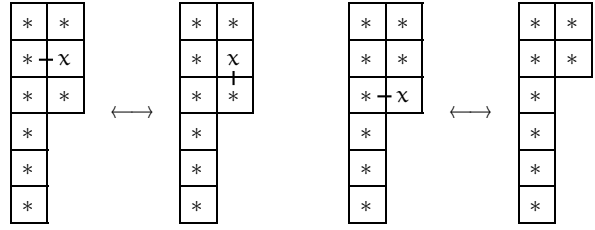
\begin{figure}[H]
\centering
\begin{tikzpicture}[scale=.75]
 \node[scale=.75] at (1,0) {\ydiagram{2,2,2,1,1,1}};
  \node[scale=.75] at (.65,1.6) [] (a){$*$};
  \node[scale=.75] at (1.3,1.6) [] (c){$*$};
  \node[scale=.75] at (.65,.95) [] (d){$*$};
  \node[scale=.75] at (.65,.3) [] (f){$*$};
  \node[scale=.75] at (.65,-.35) [] (b){$*$};
  \node[scale=.75] at (.65,-1) [] (e){$*$};
  \node[scale=.75] at (1.3,.95) [] (x){$x$};
  \node[scale=.75] at (1.3,.3) [] (h){$*$};
  \node[scale=.75] at (.65,-1.65) [] (i){$*$};
  \draw [thick](.85,.95)--(1.1,.95);
\node[scale=.75] at (2.5,0) {$\longleftrightarrow$};
 \node[scale=.75] at (4,0) {\ydiagram{2,2,2,1,1,1}};
  \node[scale=.75] at (3.65,1.6) [] (a){$*$};
  \node[scale=.75] at (4.3,1.6) [] (c){$*$};
  \node[scale=.75] at (3.65,.95) [] (d){$*$};
  \node[scale=.75] at (3.65,.3) [] (f){$*$};
  \node[scale=.75] at (3.65,-.35) [] (b){$*$};
  \node[scale=.75] at (3.65,-1) [] (e){$*$};
  \node[scale=.75] at (4.3,.95) [] (x){$x$};
  \node[scale=.75] at (4.3,.3) [] (h){$*$};
  \node[scale=.75] at (3.65,-1.65) [] (i){$*$};
  \draw [thick] (4.3,.5)--(4.3,.75);
 \node[scale=.75] at (7,0) {\ydiagram{2,2,2,1,1,1}};
  \node[scale=.75] at (6.65,1.6) [] (a){$*$};
  \node[scale=.75] at (7.3,1.6) [] (c){$*$};
  \node[scale=.75] at (6.65,.95) [] (d){$*$};
  \node[scale=.75] at (6.65,.3) [] (f){$*$};
  \node[scale=.75] at (6.65,-.35) [] (b){$*$};
  \node[scale=.75] at (6.65,-1) [] (e){$*$};
  \node[scale=.75] at (7.3,.95) [] (x){$*$};
  \node[scale=.75] at (7.3,.3) [] (h){$x$};
  \node[scale=.75] at (6.65,-1.65) [] (i){$*$};
    \draw [thick](6.85,.3)--(7.1,.3);
\node[scale=.75] at (8.5,0) {$\longleftrightarrow$};
 \node[scale=.75] at (10,0) {\ydiagram{2,2,1,1,1,1}};
  \node[scale=.75] at (9.65,1.6) [] (a){$*$};
  \node[scale=.75] at (10.3,1.6) [] (c){$*$};
  \node[scale=.75] at (9.65,.95) [] (d){$*$};
  \node[scale=.75] at (9.65,.3) [] (f){$*$};
  \node[scale=.75] at (9.65,-.35) [] (b){$*$};
  \node[scale=.75] at (9.65,-1) [] (e){$*$};
  \node[scale=.75] at (10.3,.95) [] (x){$*$};
  \node[scale=.75] at (9.65,-1.65) [] (i){$*$};
\end{tikzpicture}
\end{figure}
On the left, we observe two types of tabloids which are canceled out by the first sign-reversing involution in the proof. On the right, we see how tabloids in \\
$\mathcal{T}_{(2^3,1^3),GN_{6,2}\cup P_1}$ with $x$ in the bottom cell of the second column are mapped to tabloids in $\mathcal{T}_{(2^2,1^4),GN_{6,2}}$.
\end{example}

We now fix some notation for a certain type of Schur coefficient which satisfies \svw{the convenient recurrence relation that follows.} 

\begin{notation}
\label{notation:spec_coeff}
We let
\[f(C,D)=\begin{cases} [s_{(2^C,1^D)}]X_{GN_{C+D,C}}&\qquad\text{if }C,D\ge 0\\
0 &\qquad\text{otherwise}.
\end{cases}\] We note that
\[
\ell((2^C,1^D))=C+D=\#\text{vertices in the body of }GN_{C+D,C}.
\]
This equality makes \svwj{it easier to count} the tabloids  with a pendant-last labeling.

\end{notation}

\begin{lemma}
\label{lemma:rec_for_spec_coeff}
We have that
\[
f(C,D)=Cf(C-1,D)+Df(C,D-1)
\]
for $C,D\ge 1$.
\end{lemma}
\begin{proof}
We count tabloids using a pendant-last labeling. All $C+D$ vertices from the body of $GN_{C+D,C}$ must be in separate rim hooks and are labeled minimally. Therefore, they all lie in the first column (which accordingly has no $N$-steps).

Consider the vertex in the \svwj{bottom} cell in the first column, which we call $x$. Since $D\ge 1$, this vertex is in the tail of the diagram and is thus in a rim hook of length $1$. If $x$ is an anchor, we have a sign-preserving bijection
\[
\mathcal{T}_{(2^C,1^D),GN_{C+D,C}}^{x}\to \mathcal{T}_{(2^C,1^{D-1}),GN_{C+D-1,C-1}\cup P_1},
\]
obtained by removing the \svwj{bottom} cell in the first column. Indeed, removing an anchor detaches one of the pendant vertices, yielding a generalized net with one less pendant and one less vertex in the body, along with \svwj{an isolated vertex.}

Similarly, if $x$ is a buoy, we have a sign-preserving bijection
\[
\mathcal{T}_{(2^C,1^D),GN_{C+D,C}}^{x}\to \mathcal{T}_{(2^C,1^{D-1}),GN_{C+D-1,C}},
\]
obtained by removing the \svwj{bottom} cell in the first column. Removing a buoy results in a generalized net with one less vertex in the body.

Since there are $C$ anchors and $D$ buoys, these bijections result in the relation
\[
f(C,D)=C[s_{(2^C,1^{D-1})}]X_{GN_{C+D-1,C-1}\cup P_1}+D[s_{(2^C,1^{D-1})}]X_{GN_{C+D-1,C}}.
\]
Applying Lemma \ref{lemma:get_rid_of_singleton}, we obtain
\begin{align*}
f(C,D)&=C[s_{(2^{C-1},1^{D})}]X_{GN_{C+D-1,C-1}}+D[s_{(2^C,1^{D-1})}]X_{GN_{C+D-1,C}}\\
&=Cf(C-1,D)+Df(C,D-1).
\end{align*}
\end{proof}

We now prove that the coefficients $[s_{(2^n)}]X_{GN_{n,n}}$ satisfy a nonnegative formula.

\begin{proposition}
\label{prop:regular_tailless_form}
We have that
\[
f(n,0)=\begin{cases} n! & \text{if $n$ is even}\\
0 & \text{otherwise}
\end{cases}
\]
for $n\ge 1$.
\end{proposition}
\begin{proof}
We have that
\[
f(n,0)=[s_{(2^n)}]X_{GN_{n,n}}.
\]
The desired formula holds for $n=1$ since there are no valid SRH $GN_{1,1}$-tabloids of shape $(2)$. Likewise, the formula holds for $n=2$ since there are two horizontal rim hooks of length $2$ which may be ordered $2$ ways. We proceed by induction and assume the formula holds for $f(n-2,0)$ for some $n\ge 3$.

We count tabloids using a pendant-last labeling so that, as in the proof of Lemma \ref{lemma:only_tailless}, the first column must be filled with all the vertices from the body of the graph. For any $T\in \mathcal{T}_{(2^n),GN_{n,n}}$, we consider the \svwj{bottom} rim hook which necessarily starts with an $E$-step then includes $N$-steps up the second column spanning $i$ cells where $i$ ranges from $1$ to $n-1$. Note this rim hook cannot span all $n$ cells in the second column since there are no stable $n+1$-sets in $GN_{n,n}$.

For each choice of a \svwj{bottom} rim hook spanning $i$ cells in the second column, we consider a mapping in which this rim hook is removed, leaving the rest of the tabloid unchanged. In each case, this results in a bijection from the subset of tabloids with the given rim hook to the set
\[\mathcal{T}_{(\svw{2^{n-i},1^{i-1}}),GN_{n-1,n-i-1}\cup P_1}.\]
Indeed, each such map can be inverted by simply adding the given rim hook back to the tabloid. Under such a mapping, the sign of the tabloid changes by a factor of $(-1)^{i-1}$ since there are $i-1$ $N$-steps being removed. There are $n$ choices for the anchor at the beginning of the rim hook, then $\binom{n-1}{i}$ choices for the $i$ pendants in the rim hook (their order is determined by the labeling and the pendant corresponding to the anchor is excluded).

Accordingly, via these bijections, we obtain the formula
\begin{align*}
f(n,0)=n\sum_{i=1}^{n-1}\binom{n-1}{i}(-1)^{i-1}[s_{(2^{n-i},1^{i-1})}]X_{GN_{n-1,n-i-1}\cup P_1}.
\end{align*}
By Lemma \ref{lemma:get_rid_of_singleton}, we can convert this to
\begin{align}
\begin{split}
\label{eq:bot_hook_remov_formula}
f(n,0)&=n\sum_{i=1}^{n-1}\binom{n-1}{i}(-1)^{i-1}f(n-i-1,i).\\
\end{split}
\end{align}
We then have
\begin{align*}
f(n,0)&=n(-1)^{n-2}f(0,n-1)+n\sum_{i=1}^{n-2}\binom{n-1}{i}(-1)^{i-1}f(n-i-1,i)\\
&=(-1)^{n}n!+n\sum_{i=1}^{n-2}\binom{n-1}{i}(-1)^{i-1}f(n-i-1,i)
\end{align*}
since $f(0,n-1)=(n-1)!$ as it simply counts tabloids of shape $(1^{n-1})$ with no $N$-steps. We then apply Lemma \ref{lemma:rec_for_spec_coeff} and Equation \ref{eq:bot_hook_remov_formula} to obtain
\begin{align*}
f(n,0)&=(-1)^{n}n!+n\sum_{i=1}^{n-2}\binom{n-1}{i}(-1)^{i-1}\big((n-i-1)f(n-i-2,i)+if(n-i-1,i-1)\big)\\
&=(-1)^{n}n!+n(n-1)\sum_{i=1}^{n-2}\frac{(n-2)!}{(n-2-i)!i!}(-1)^{i-1}f(n-i-2,i)\\
&+n(n-1)\sum_{i=1}^{n-2}\frac{(n-2)!}{(n-1-i)!(i-1)!}(-1)^{i-1}f(n-i-1,i-1)\\
&=(-1)^{n}n!+nf(n-1,0)+n(n-1)\sum_{j=0}^{n-3}\frac{(n-2)!}{(n-2-j)!j!}(-1)^{j}f(n-j-2,j)\\
&=(-1)^{n}n!+nf(n-1,0)+n(n-1)f(n-2,0)\\
&-n(n-1)\sum_{j=1}^{n-2}\binom{n-2}{j}(-1)^{j-1}f(n-j-2,j)+n(n-1)(-1)^{n-3}f(0,n-2)\\
&=(-1)^{n}n!+nf(n-1,0)+n(n-1)f(n-2,0)-nf(n-1,0)+(-1)^{n-1}n!\\
&=n(n-1)f(n-2,0)\\
&=\begin{cases} n(n-1)\cdot (n-2)! &\text{if $n$ is even}\\
n(n-1)\cdot 0 &\text{if $n$ is \svwj{odd,}}
\end{cases}
\end{align*}
so the desired formula holds by induction.
\end{proof}


We now prove the main theorem of this section.

\begin{theorem}
\label{thm:gen_nets_are_s_pos}
All generalized nets $GN_{n,m}$ are Schur-positive for $n\ge 1$ and $n\ge m \ge 0$.
\end{theorem}
\begin{proof}
We proceed by induction on the number of vertices.

If $m=0$, $1$, or $2$, $GN_{n,m}$ is a claw-free incomparability graph, as discussed in \cite{GebSag01}. Thus, these graphs are Schur-positive by Theorem \ref{thm:claw_free_incs}.

Assume $GN_{n,m-1}$ are Schur-positive for all $n\ge m-1$ for some $m\ge 3$. We first show $GN_{m,m}$ is Schur-positive. Consider any $\lambda=(\lambda_1,\dots,\lambda_k)$ such that $[s_{\lambda}]X_{GN_{m,m}}\ne 0$. If $\lambda_k\ne 1$, by Lemma \ref{lemma:only_tailless}, $\lambda=(2^m)$. Then, by Proposition \ref{prop:regular_tailless_form},
\[
[s_{(2^m)}]X_{GN_{m,m}}> 0.
\]Otherwise, \svwj{we} may assume $\lambda_k=1$. By Proposition \ref{prop:rec_for_most_coef},
\[
[s_{\lambda}]X_{GN_{m,m}}=m\xi(\lambda\setminus 1,GN_{m-1,m-1}\cup P_1)+m\xi(\lambda\setminus 1^2,GN_{m-1,m-1})\ge 0
\]
where the righthand side is nonnegative by the inductive hypothesis (note disjoint unions of Schur-positive graphs are Schur-positive by Prop. 2.3 in \cite{stan_95}).


Assume $GN_{n-1,m}$ is Schur-positive for some $n\ge m+1$. Consider any $\lambda=(\lambda_1,\dots,\lambda_k)$ such that $[s_{\lambda}]X_{GN_{n,m}}\ne 0$. By Lemma \ref{lemma:only_tailless}, $\lambda_k=1$ since $n\ne m$. We have
\begin{align*}
[s_{\lambda}]X_{GN_{n,m}}&=m\xi(\lambda\setminus 1,GN_{n-1.m-1}\cup P_1)+(n-m)\xi(\lambda\setminus 1,GN_{n-1,m})\\
&+m\xi(\lambda\setminus 1^2,GN_{n-1,m-1})\ge 0
\end{align*}
by Proposition \ref{prop:rec_for_most_coef} and \svwj{by} the two inductive hypotheses, \svw{and we are done.}\end{proof}


\section{Generalized Spiders}
\label{sect:gen_spiders}

In this section, we address a larger family of graphs known as generalized spiders, which includes all generalized nets. We employ Proposition \ref{prop:rec_algorithm} to derive a recurrence relation for the Schur coefficients of these graphs.

\begin{definition}
Let $\lambda=(\lambda_1,\dots,\lambda_k)$ be a partition. A \textit{generalized spider} $GS_{n,\lambda}$ for $\svwj{n\ge 3}, n\ge k\ge 0$, is a complete graph $K_n$ with paths of \svwj{lengths} $\lambda_1,\dots,\lambda_k$ appended to distinct vertices in the complete graph. As with generalized nets, we refer to degree $n-1$ vertices as \textit{buoys}, degree $n$ vertices as \textit{anchors}, and the subgraph consisting of all buoys and anchors as the \textit{body}.
\end{definition}

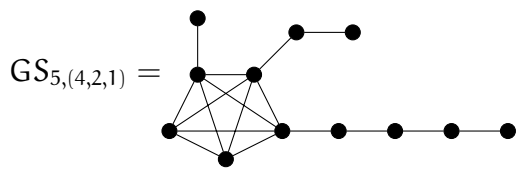
\begin{example}\,
\svwj{The following generalized spider has 3 anchors and 2 buoys together giving 5 vertices in the body.}
\begin{figure}[H]
\centering
\begin{tikzpicture}[scale=.75]
\node[] at (-2,1.5) {$GS_{5,(4,2,1)}=$};
\node[circle,draw,scale=.5,fill=black] at (1,1.5) (2) {};
\node[circle,draw,scale=.5,fill=black] at (1.5,.5) (3) {};
\node[circle,draw,scale=.5,fill=black] at (-.5,.5) (5) {};
\node[circle,draw,scale=.5,fill=black] at (0,1.5) (1) {};
\node[circle,draw,scale=.5,fill=black] at (.5,0) (4) {};
\node[circle,draw,scale=.5,fill=black] at (1.75,2.25) (7) {};
\node[circle,draw,scale=.5,fill=black] at (0,2.5) (6) {};
\node[circle,draw,scale=.5,fill=black] at (2.5,.5) (9) {};
\node[circle,draw,scale=.5,fill=black] at (3.5,.5) (10) {};
\node[circle,draw,scale=.5,fill=black] at (4.5,.5) (11) {};
\node[circle,draw,scale=.5,fill=black] at (5.5,.5) (12) {};
\node[circle,draw,scale=.5, fill=black] at (2.75,2.25) (8) {};
\draw (1)--(2)--(3)--(4)--(5)--(1)--(3)--(5)--(2)--(4)--(1);
\draw (1)--(6);
\draw (2)--(7)--(8);
\draw (3)--(9)--(10)--(11)--(12);
\end{tikzpicture}
\end{figure}
\end{example}

We have that generalized nets $GN_{n,m}$ are generalized spiders $GS_{n,(1^m)}$ for $\svwj{n\ge 3}, n\ge m \ge 0$.

Recall that a \textit{spider} is a tree with exactly one vertex of degree $3$ or greater. Recall also that the \textit{line graph} of a graph $G$ is the graph $L(G)$ with vertices corresponding to the edges of $G$ such that: two vertices of $L(G)$ are adjacent if and only if the corresponding edges of $G$ are incident to the same vertex. It is shown in \cite{spiders_kin} that a graph $G$ is a generalized spider with a body of size $n\ge 3$ if and only if $G$ is the line graph of some spider. \svwj{It is known that all line graphs are claw-free {\cite{Beineke},} and therefore generalized spiders are as well for $n\ge 3$.}




In this section, we will focus exclusively on the simplest generalized spiders which are neither generalized nets nor paths, that is, the family $GS_{n,(2,1^{m-1})}$ where $n\ge m\ge 1$ and $n\ge 3$. We refer to all vertices outside of the body as \textit{pendants}, the unique pendant nonadjacent to all anchors as the \textit{special pendant}, the other pendants as \textit{regular pendants}, and the unique anchor connected to the path of length $2$ with the special pendant as the \textit{special anchor}. We begin by proving a lemma that mirrors Lemma \ref{lemma:tail_cant_be_pend_filled} but applies to generalized spiders and requires the additional assumption that $\lambda_{k-1}=1$.

\begin{lemma}
\label{lemma:not_pend_filled_gen_spiders}
Consider a partition $\lambda=(\lambda_1,\dots,\lambda_k)$ and assume $\lambda_k=\lambda_{k-1}=1$.  For any $T\in\mathcal{T}_{\lambda, GS_{n,(2,1^{m-1})}}$ where $n\ge m\ge 1$ \svwj{and $n\geq 3$}, we have that, regardless of the choice of labeling, the tail cannot contain only pendants.
\end{lemma}

\begin{proof}
Assume the tail of $T$ has $h$ pendants and no vertices from the body. The head must contain $n+m+1-h$ cells to include all the other vertices. Accordingly, the head may have at most $\lfloor \frac{n+m+1-h}{2} \rfloor$ rows and at most $\lfloor \frac{n+m+1-h}{2} \rfloor$ distinct rim hooks since every rim hook must intersect the first column. We then have
\[
\#\text{vertices from the body in head}\le \# \text{rim hooks in the head} \le \left\lfloor \frac{n+m+1-h}{2} \right\rfloor < n.
\]
The first inequality follows since every \svw{body} vertex \svw{in} the head must be in its own rim hook \svw{by definition.} The last inequality follows since $m\le n$ and $h\ge 2$. Hence, at least one of the $n$ vertices from the body must be in the tail, contradicting the initial assumption.
\end{proof}

We now  \svw{derive} the key recurrence relation of this section which is an analogue to Proposition \ref{prop:rec_for_most_coef}.

\begin{proposition}
\label{prop:spi_rec}
Let $\lambda=(\lambda_1,\dots,\lambda_k)$ be a partition and assume $\lambda_{k}=\lambda_{k-1}=1$. We have
\begin{align}
\begin{split}
\label{eq:spi_rec}
\xi (\lambda,GS_{n,(2,1^{m-1})})&=(m-1)\xi(\lambda\setminus 1,GS_{n-1,(2,1^{m-2})}\cup P_1)+(n-m)\xi(\lambda\setminus 1,GS_{n-1,(2,1^{m-1})})\\
&+\xi(\lambda\setminus 1,GS_{n-1,(1^{m-1})\cup P_2})+(m-1)\xi(\lambda\setminus 1^2,GS_{n-1,(2,1^{m-2})})\\
&+\xi(\lambda\setminus 1^3,GS_{n-1,(1^{m-1})})
\end{split}
\end{align}
for $n\ge 3$, $n\ge m\ge 2$.
\end{proposition}

\begin{proof}
We denote the set of all anchors of $GS_{n,(2,1^{m-1})}$ \svwj{\emph{except} the special anchor} by $A$, the set of all buoys of $GS_{n,(2,1^{m-1})}$ by $B$, and the set of all regular pendants of $GS_{n,(2,1^{m-1})}$ by $P$, the special pendant by $\tilde{p}$, and the special anchor by $\tilde{a}$. We count tabloids with a pendant-first labeling where the special pendant $\tilde{p}$ has the minimum label. \svwj{Note that consequently the pendant adjacent to the special pendant has the second minimal label.} We observe that
\begin{align}
\begin{split}
\label{eq:spi_rec_tab_partition}
\mathcal{T}_{\lambda, GS_{n,(2,1^{m-1})}}=&\mathcal{T}_{\lambda,GS_{n,(2,1^{m-1})}}^A\sqcup \mathcal{T}_{\lambda,GS_{n,(2,1^{m-1})}}^B\sqcup \mathcal{T}_{\lambda,GS_{n,(2,1^{m-1})}}^P\\
&\sqcup \mathcal{T}_{\lambda GS_{n,(2,1^{m-1})}}^{\tilde{p}}\sqcup \mathcal{T}_{\lambda,GS_{n,(2,1^{m-1})}}^{\tilde{a}}.
\end{split}
\end{align}

In the case where an anchor $a\in A$, buoy $b\in B$, or the special anchor $\tilde{a}$ is in the bottom row, there is no $N$-step up from the cell since we are using a pendant-first labeling. Hence, we can map these tabloids to tabloids for which the bottom cell is removed. These are sign-preserving bijections
\begin{align}
\begin{split}
\label{eq:spi_rec_biject}
\mathcal{T}_{\lambda,GS_{n,(2,1^{m-1})}}^a&\cong \mathcal{T}_{\lambda\setminus 1 ,GS_{n-1,(2,1^{m-2})}\cup P_1},\qquad\qquad \mathcal{T}_{\lambda,GS_{n,(2,1^{m-1})}}^b\cong \mathcal{T}_{\lambda\setminus 1,GS_{n-1,(2,1^{m-1})}}\\
&\andd\mathcal{T}_{\lambda,GS_{\svwj{n},(2,1^{m-1})}}^{\tilde{a}}\cong \mathcal{T}_{\lambda\setminus 1,GS_{\svwj{n-1},(1^{m-1})\cup P_2}},
\end{split}
\end{align}
since removing an anchor detaches a pendant, decreasing the number of body vertices and the number of pendants by one, removing a buoy just decreases the number of body vertices by one, and removing the special anchor detaches a path of length $2$, also decreasing the number of body vertices by $1$.

From Equation \ref{eq:spi_rec_tab_partition} and Equation \ref{eq:spi_rec_biject}, we then have
\begin{align}
\begin{split}
\label{eq:spi_rec_inc_form}
\xi (\lambda,GS_{n,(2,1^{m-1})})&=(m-1)\cdot\sum_{T\in \mathcal{T}_{\lambda\setminus 1 ,GS_{n-1,(2,1^{m-2})}\cup P_1}}\mathrm{sgn}(T)\\
&+(n-m)\cdot\sum_{T\in \mathcal{T}_{\lambda\setminus 1,GS_{n-1,(2,1^{m-1})}}}\mathrm{sgn}(T)\\
&+\sum_{T\in \mathcal{T}_{\lambda\setminus 1 ,GS_{n-1,(1^{m-1})}\cup P_2}} \svw{\mathrm{sgn}(T)}+\sum_{T\in \mathcal{T}_{\lambda,GS_{n,(2,1^{m-1})}}^P}\mathrm{sgn}(T)+\sum_{T\in \mathcal{T}_{\lambda,GS_{n,(2,1^{m-1})}}^{\tilde{p}}}\mathrm{sgn}(T)\\
&=(m-1)\xi(\lambda\setminus 1,GS_{n-1,(2,1^{m-2})}\cup P_1)+(n-m)\xi(\lambda\setminus 1,GS_{n-1,(2,1^{m-1})})\\
&+\xi(\lambda\setminus 1,GS_{n-1,(1^{m-1})}\cup P_2)+\sum_{T\in \mathcal{T}_{\lambda,GS_{n,(2,1^{m-1})}}^P}\mathrm{sgn}(T)\\
&+\sum_{T\in \mathcal{T}_{\lambda,GS_{n,(2,1^{m-1})}}^{\tilde{p}}}\mathrm{sgn}(T)
\end{split}
\end{align}
since there are $m-1$ anchors, $n-m$ buoys, and $1$ special anchor $\tilde{a}$ in the graph.

Consider the subset of tabloids $\mathcal{S}^p$ in  $\mathcal{T}^P_{\lambda, GS_{n,(2,1^{m-1})}}$ such that a certain pendant $p\in P$ is in the bottom cell, directly below its anchor. Since there are no $N$-steps up from the anchor, we map these tabloids to tabloids where the bottom two cells are removed. We can accordingly obtain sign-preserving bijections 
\[
\mathcal{S}^p\cong \mathcal{T}_{\lambda\setminus 1^2,GS_{n-1,(2,1^{m-2})}},
\]
for each pendant $p$, since removing a pendant and corresponding anchor lowers the number of pendants and the number of body vertices by one.

We can thus count the signs of tabloids in each $\mathcal{S}^p$ by adding the term
\begin{align*}
(m-1)\sum_{T\in\mathcal{T}_{\lambda\setminus 1^2,GS_{n-1,(2,1^{m-2})}}}\mathrm{sgn}(T)=(m-1)\xi(\lambda\setminus 1^2,GS_{n-1,(2,1^{m-2})})
\end{align*}
to the sum we are computing (we have a factor of $m-1$ since there are $m-1$ pendants $p$ in $P$).

Now, consider the subset $\mathcal{S}^{\tilde{p}}$of tabloids in  $\mathcal{T}^{\tilde{p}}_{\lambda, GS_{n,(2,1^{m-1})}}$ such that the special pendant $\tilde{p}$ is in the bottom cell, directly below its unique adjacent pendant $p'$, which is below the special anchor $\tilde{a}$. Since there are no $N$-steps up from any of these three vertices, we map these tabloids to tabloids where the bottom three cells are removed. We can accordingly obtain a sign-preserving bijection 
\[
\mathcal{S}^{\tilde{p}}\cong\mathcal{T}_{\lambda\setminus 1^3,GS_{n-1,(1^{m-1})}},
\]
since removing the special pendant $\tilde{p}$, its adjacent pendant $p'$, and the special anchor $\tilde{a}$ removes the path of length $2$ containing $\tilde{p}$ and $p'$ and reduces the number of body vertices by one.

We can thus count the signs of the tabloids in $\mathcal{S}^{\tilde{p}}$ by adding the term
\begin{align*}
\sum_{T\in\mathcal{T}_{\lambda\setminus 1^3,GS_{n-1,(1^{m-1})}}}\mathrm{sgn}(T)=\xi(\lambda\setminus 1^3,GS_{n-1,(1^{m-1})}).
\end{align*}

Let 
\[
    \mathcal{S}= \Bigg(\mathcal{T}^{P}_{\lambda, GS_{n,(2,1^{m-1})}}\setminus \bigsqcup_{p \in P}\mathcal{S}^p\Bigg)
    \bigsqcup \Bigg(\mathcal{T}^{\tilde{p}}_{\lambda, GS_{n,(2,1^{m-1})}} \setminus \mathcal{S}^{\tilde{p}}\Bigg)
\]
denote the set of all tabloids which remain to be counted.

We partition $\mathcal{S}$ as
\[
\mathcal{S}=\mathcal{S}^1\sqcup \mathcal{S}^2
\]
where $\mathcal{S}^1$ includes all tabloids $T$ in $\mathcal{S}$ such that the special pendant $\tilde{p}$ does not appear in $\text{tl}(T)$ and $\mathcal{S}^2=\mathcal{S}\setminus \mathcal{S}^1$. We then partition $\mathcal{S}^1$ such that
\[
\mathcal{S}^1=\bigsqcup_{H}\mathcal{S}^1_H,
\]
where the disjoint union spans over all possible heads of tabloids $T\in \mathcal{S}^1$ and each $\mathcal{S}^1_H\subseteq \mathcal{S}^1$ is exactly the subset such that
\[
\mathrm{hd}(T)=\mathrm{hd}(T')=H
\]
for all $T,T'\in\mathcal{S}^1_H$.

For any $\mathcal{S}^1_H$, let $\mathcal{P}_H$ and $\mathcal{U}_H$ respectively denote the set of pendants and the set of body vertices appearing in the tail of every $T\in\mathcal{S}^1_H$. By Lemma \ref{lemma:not_pend_filled_gen_spiders}, the \svwj{premise} that the special pendant $\tilde{p}$ does not appear in the tail of any tabloid in $\mathcal{S}^1$, and the definition of a generalized spider, Conditions (I) and (II) of Proposition \ref{prop:rec_algorithm} are satisfied for $\mathcal{P}_H$ and $\mathcal{U}_H$.

Moreover, all $u,u'\in U$ are all adjacent so Condition (i) of Proposition \ref{prop:rec_algorithm} is also satisfied. Lastly, recall $\mathcal{S}^1$ does not include any tabloids with a pendant in the bottom cell directly below its anchor. Therefore, Condition (ii) of Proposition \ref{prop:rec_algorithm} holds as well. We conclude
\[
\sum_{T\in\mathcal{S}^1_H}\mathrm{sgn}(T)=0
\qquad\text{for all $H$ so}\qquad
\sum_{T\in\mathcal{S}^1}\mathrm{sgn}(T)=0.
\]

Now, consider the subset of all tabloids in $\mathcal{S}^2$ such that the special pendant $\tilde{p}$ appears in the tail below some vertex other than its adjacent pendant $p'$. Since $\tilde{p}$ is nonadjacent to all other vertices and labeled minimally, we can define a sign-reversing involution on these tabloids by adding an $N$-step up from $\tilde{p}$ if it is not there and removing an $N$-step up from $\tilde{p}$ if it is there. Thus, the signs of all these tabloids cancel out.

Every remaining tabloid in $\mathcal{S}^2$ satisfies the property that $\tilde{p}$ appears below its adjacent pendant $p'$ in $\text{tl}(T)$. Since there are no $N$-steps up from $\tilde{p}$ in this case, we map these tabloids to the tabloids where $p'$ and $\tilde{p}$ have been combined into one vertex $\tilde{p}'$ (and their cells in the tabloid have been combined). We can accordingly obtain a sign-preserving injection from this subset of $\mathcal{S}^2$ into
\[
\mathcal{T}_{\lambda\setminus 1,GS_{n,(1^{m})}}=\mathcal{T}_{\lambda\setminus 1,GN_{n,m}},
\]
since this map removes the path of length $2$ with the special pendant, replacing it with a regular pendant, and thus obtaining a generalized net.

The image of this map includes exactly all SRH $GN_{n,m}$-tabloids $T$ of shape $\lambda\setminus 1$ such that the bottom cell in $T$ is filled by a pendant which is nonadjacent to the vertex above it. This is because all the tabloids with a regular pendant in the bottom cell below its anchor have been counted as the sets $\mathcal{S}^p$. Likewise, the SRH $GN_{n,m}$-tabloids which have $\tilde{p}'$ in the bottom cell below its anchor correspond to SRH $GS_{n,(2,1^{m-1})}$-tabloids with $\tilde{p}$ in the bottom cell below $p'$ \svwj{(with label 2)} below the special anchor $\tilde{a}$. These tabloids have also been counted as the set $\mathcal{S}^{\tilde{p}}$.

As described in the proof of Proposition \ref{prop:rec_for_most_coef}, we can now apply Proposition \ref{prop:rec_algorithm} to conclude
\[
\sum_{T\in \mathcal{S}^2}\text{sgn}(T)=0,
\]
and thus that Equation \ref{eq:spi_rec} holds.\end{proof}

Proving Schur-positivity for the family of claw-free graphs $GS_{n,(2,1^{m-1})}$ would be a natural extension of Theorem \ref{thm:gen_nets_are_s_pos} since this family contains graphs which are very similar to generalized nets and are still conjectured to be Schur-positive by Conjecture \ref{conj:claw_free}. However, unlike Proposition \ref{prop:rec_for_most_coef}, Proposition \ref{prop:spi_rec} requires the additional assumption that the given partition has a tail of length at least $2$ (or, equivalently, $\lambda_{k-1}=1$). This is an obstacle that necessitates other arguments for showing the nonnegativity of the coefficients 
\[
[s_{(3,2^{n-1})}]X_{GS_{n,(2,1^{n-1})}},\qquad [s_{(2^{n},1)}]X_{GS_{n,(2,1^{n-1})}},\andd [s_{(2^{n})}]X_{GS_{n,(2,1^{n-2})}}.
\]
Note that these are exactly the appearing nonzero coefficients which are not covered by Lemma \ref{lemma:not_pend_filled_gen_spiders}. If these three coefficients are always nonnegative, then all generalized spiders $GS_{n,\svwj{(2,1^{m-1})}}$ are Schur-positive.


\end{document}